\documentclass[11pt,reqno]{amsart}

\numberwithin{equation}{section}
\usepackage{times}
\usepackage{amsmath,amsfonts,amstext,amssymb,amsbsy,amsopn,amsthm,eucal}
\usepackage{txfonts}
\usepackage{dsfont}
\usepackage{graphicx}   
\usepackage{hyperref}
\usepackage{accents}
\usepackage{enumerate}
\usepackage{xcolor}
\usepackage{verbatim}
\usepackage{esint}

\usepackage[normalem]{ulem}
\usepackage{cancel}

\newcommand{\RR}{\mathds{R}}

\newcommand{\cC}{\mathcal{C}}

\newcommand{\cH}{\mathcal{H}}

\newcommand{\cL}{\mathcal{L}}

\newcommand{\cM}{\mathcal{M}}
\newcommand{\cN}{\mathcal{N}}

\newcommand{\cS}{\mathcal{S}}

\newcommand{\cV}{\mathcal{V}}

\newtheorem{theorem}[equation]{Theorem}

\newtheorem{proposition}[equation]{Proposition}
\newtheorem{lemma}[equation]{Lemma}
\newtheorem{corollary}[equation]{Corollary}
\theoremstyle{definition}
\newtheorem{definition}[equation]{Definition}

\theoremstyle{remark}
\newtheorem{remark}[equation]{Remark}
\theoremstyle{remark}

\theoremstyle{remark}

\theoremstyle{remark}\newtheorem{conjecture}[equation]{Conjecture}
\theoremstyle{remark}

\begin{document}

\thanks{}
\thanks{}

\title[]{The Bounded Diameter Conjecture and Sharp Geometric Estimates for Mean Curvature Flow}

\author{Yiqi Huang and Wenshuai Jiang}
\address[Yiqi Huang]{Department of Mathematics, MIT, 77 Massachusetts Avenue, Cambridge, MA 02139-4307, USA}
\address[Wenshuai Jiang]{School of Mathematical Sciences, Zhejiang University, Hangzhou 310058, China; 
School of Mathematics, Institute for Advanced Study, 1 Einstein Drive, Princeton, New Jersey, 08540 USA}

 \email{yiqih777@mit.edu}
 \email{wsjiang@zju.edu.cn, wsjiang@ias.edu}

\maketitle

\begin{abstract}

We show that the intrinsic diameter of mean curvature flow in $\mathbb{R}^3$ is uniformly bounded as one approaches the first singular time $T$. This confirms the \emph{bounded diameter conjecture} of Haslhofer~\cite{Has21,Has25}. In addition, we establish several sharp quantitative estimates: the second fundamental form $A$ has uniformly bounded $L^1$-norm on each time slice, $A$ belongs to the weak $L^3$ space on the space-time region, and the singular set $\mathcal{S}$ has finite $\mathcal{H}^1$-Hausdorff measure. All of the results are optimal due to the marriage ring example and our results do not require any convexity assumptions on the surfaces. Furthermore, our arguments extend naturally to flows through singularities, yielding the same sharp estimates.

\end{abstract}


\section{Introduction}

A smooth family of closed embedded surfaces $\cM \subset \RR^3 \times [0,T)$ evolves by mean curvature flow (\cite{Bra78,Hu84}) if the normal velocity at each point is given by the mean curvature vector. Finite-time singularities are inevitable, and a central problem in the study of mean curvature flow is to understand the structure of the singular set and the geometry of the flow near singularities. 

In this paper, our first main result establishes a uniform bound for the intrinsic diameter of the flow, thereby confirming a conjecture of Haslhofer:

\begin{conjecture}[Haslhofer \cite{Has21,Has25}]\label{c:bound diam}
    The intrinsic diameter of the mean curvature flow $\cM$ in $\RR^3$ remains uniformly bounded as $t \nearrow T$, the first singular time. 
\end{conjecture}

This conjecture was inspired by an analogous conjecture of Perelman in three-dimensional Ricci flow \cite{Pe02}. For mean curvature flow, the extrinsic diameter is trivially bounded by the avoidance principle, but the intrinsic diameter is far subtler: near singularities, the evolving surface may develop long, spiraling, or fractal-like regions whose intrinsic geometry could, a priori, degenerate.

By Huisken’s monotonicity formula \cite{Hu90}, singularities of mean curvature flow are modeled on self-shrinkers. Previous attempts to address Conjecture~\ref{c:bound diam} have relied on precise classifications of possible singularity models. The first such result was obtained by Gianniotis and Haslhofer \cite{GH20}, who proved the conjecture in the mean convex case. Thanks to the deep work of White ~\cite{Wh00,Wh03}, all tangent flows in this setting are multiplicity-one round cylinders, see \cite{Hu93,HS99a,HS99b,SW09,An12,CM12,HK17a}. Combining this classification with the Łojasiewicz–Simon inequality for cylindrical shrinkers established in Colding and Minicozzi's seminal work \cite{CM15} (also \cite{CIM15}) and the canonical neighborhood theorem from Haslhofer-Kleiner \cite{HK17a,HK17b}, they ruled out the formation of “fractal tubes" near cylindrical singularities and obtained a uniform diameter bound.

For general flows without mean convexity, however, the singularities are considerably more intricate. By the classification theorems of Wang \cite{Wa16b} and Bamler–Kleiner \cite{BK23}, every non-compact shrinker in $\mathbb{R}^3$ has either conical or cylindrical ends. Du \cite{Du21} recently verified Conjecture~\ref{c:bound diam} under the additional assumption that the flow only encounters multiplicity-one neck singularity or singularity with conical end, using the mean convex neighborhood theorem by Choi-Haslhofer-Hershkovits \cite{CHH22} and the uniqueness of tangent flows at conical singularities proved by Chodosh and Schulze \cite{CS21}.

A recent breakthrough of Bamler and Kleiner \cite{BK23} resolved the higher-multiplicity issue in mean curvature flow in $\RR^3$. Nevertheless, the structure of singularities with cylindrical ends remains largely mysterious. This connects to a long-standing question posed by Ilmanen \cite{Il95}: Is the round cylinder the only complete embedded shrinker in $\mathbb{R}^3$ with a cylindrical end? Wang \cite{Wa16a} proves the rigidity under the assumption of the convergence of infinite order on the cylindrical end, and constructed non-trivial incomplete examples with such ends. Apart from this, however, very little is known. As a consequence, previous approaches relying on establishing the canonical neighborhood theorems cannot address the general setting.

In this paper, we confirm the bounded diameter conjecture ~\ref{c:bound diam} in full generality. In fact, our main theorem provides much stronger information: we establish optimal curvature estimates and quantitative regularity of the singular set.

\begin{theorem}\label{t:main smooth}
    Let $\{\cM_t \subset \RR^3\}_{t\in [0,T)}$ be a smooth closed embedded mean curvature flow, then there exists some constant $C=C(\cM_0)<\infty$ such that for all $0\le t<T$
    \begin{enumerate}
        \item The intrinsic diameter of $\cM_t$ is uniformly bounded:  $$D_{int}(\cM_t) < C;$$
        \item The $L^1$-norm of the second fundamental form $A$ is uniformly bounded: $$\int_{\cM_t} |A| < C;$$
        \item The second fundamental form $A$ has bounded weak $L^3$-norm:
        $$\mu_{\cM}\big( \{ X \in \cM : |A|(X) \ge s^{-1} \} \big) \le C s^3;$$
        \item The singular set $\mathcal{S}_T$ of the final time slice $\mathcal{M}_T$ is contained in a countable union of embedded 1-dimensional Lipschitz manifolds together with countably many points. Moreover, $\cS_T$ has finite 1-dimensional Hausdorff measure:
        $$\cH^1(\cS_T) \le C.$$
    \end{enumerate}
\end{theorem}
\begin{remark}
The estimates in Theorem~\ref{t:main smooth} are sharp. For example, the marriage ring solution, which shrinks to a round circle at time $T$, realizes equality in the curvature and singular set bounds.
\end{remark}
\begin{remark}
    Our argument in fact yields the stronger Minkowski estimate: $Area(B_r(\cS_T)) \le Cr$, which refines the Hausdorff estimate above. Moreover, our argument provides an integral estimate for regularity scale, which is much stronger than the curvature estimate.    
\end{remark}

The uniform $L^1$-bound on the second fundamental form $A$ is in fact stronger than the intrinsic diameter bound. Indeed, Topping \cite{To08} established the inequality $D_{int}(\cM_t) \le C \int_{\cM_t} |H|$, which shows that controlling the mean curvature in $L^1$ already implies a diameter bound.  However, we do not need to use Topping's result in our proof. In fact, we prove the diameter bound and the $L^1$ bound of $A$ simultaneously. Under the mean convexity assumption,  Gianniotis  and Haslhofer \cite{GH20} proved an $L^1$-bound on $A$; see also related work of Head \cite{He13}, Cheeger–Haslhofer–Naber \cite{CHN13}, and Du \cite{Du21}. Our result removes all convexity assumptions and establishes the same conclusion in complete generality. 

For the space–time estimates, Cheeger–Haslhofer–Naber \cite{CHN13} previously obtained the $L^p$ bound for $A$ with any $p<3$ in the mean convex setting. Without assuming convexity, we improve this to an \emph{optimal weak 
$L^3$} bound, which is sharp in view of the cylindrical self-shrinkers. It is instructive to compare this with the corresponding elliptic estimates. In \cite{NV20}, Naber-Valtorta obtained the optimal weak $L^7$ estimate for minimizing hypersurfaces, refining the earlier $L^p$ bound with $p<7$ by Cheeger-Naber \cite{CN13}. Similar optimal estimates have since been obtained for minimizing harmonic maps \cite{NV17} and related work for Yang-Mills and Einstein manifolds \cite{NV19,JN21}.  

Our results further extend beyond the first singular time, to weak formulations of mean curvature flow. There are several notions of weak mean curvature flow in the literature, including the Brakke flow and level-set formulations. Very recently, Bamler and Kleiner \cite{BK23} introduced the notion of \emph{almost regular flows}: these are well-behaved Brakke flows whose local scale functions satisfy quantitative integrability conditions. They proved that this class of flows includes most of the commonly used types of mean curvature flows “through singularities" and thus resolved Ilmanen’s multiplicity-one conjecture \cite{Il95}, extending multiplicity-one property to almost regular flows.

As our second main result, we show that both the intrinsic diameter bound and the curvature estimates established above extend verbatim to almost regular flows starting from smooth closed embedded initial data. If the surface is disconnected, we define the intrinsic diameter to be the sum of that at each connected component. Moreover, we prove that the space–time singular set $\cS$ of such flows has \emph{finite 1-dimensional parabolic Hausdorff measure}, a stronger statement than the slice estimate in Theorem~\ref{t:main smooth}.

\begin{theorem}\label{t:main almost regular}
    Let $\cM \subset \RR^3 \times [0,T]$ be a bounded almost regular mean curvature flow in $\RR^3$ starting from a smooth embedded closed surface $M_0$. Then there exists some constant $C=C(\cM)< \infty$ such that the following estimates hold for all $t\in [0,T]$:
    \begin{enumerate}
        \item The intrinsic diameter of $\cM_t \setminus\cS$ is uniformly bounded:  $$D_{int}(\cM_t\setminus\cS) < C;$$
        \item The $L^1$-norm of the second fundamental form $A$ is uniformly bounded: $$\int_{\cM_t\setminus \cS} |A| < C;$$
        \item The second fundamental form $A$ has bounded weak $L^3$-norm:
        $$\mu_{\cM}\big( \{ X \in \cM : |A|(X) \ge s^{-1} \} \big) \le C s^3;$$
        \item The singular set $\cS$ is contained in a countable union of embedded 1-dimensional Lipschitz manifold together with countably many points. Moreover, $\cS$ has finite 1-dimensional parabolic Hausdorff measure $$\cH_P^1(\cS) < C.$$
    \end{enumerate}
\end{theorem}

Previously, combining the results of \cite{CHN13,CM16,BK23} implies that the space-time singular set $\cS$ is 1-rectifiable and satisfies the dimensional bound $\dim_P(\cS) \le 1$. Under the mean convex setting, Colding-Minicozzi \cite{CM16} proved the measure finiteness around the cylindrical singularity by observing that the set of cylindrical singularity is compact. However, the compactness fails in the general case, as the cylindrical singularities may converge to non-generic singularities. Our contribution strengthens these results by showing that the $1$-dimensional measure of $\mathcal{S}$ is \emph{finite} without the any assumption on the singularity type. Moreover, our arguments yield a quantitative Minkowski estimate, which in particular implies the Hausdorff bound and provides a sharper geometric control of the singular set.

Our method is different from the previous approaches to the conjecture~ \ref{c:bound diam}, which reply on the classification of singularity models and establishing the canonical neighborhood theorems. Instead, we prove some covering theorems using regular balls to cover the slice, inside which the flow is almost flat and thus has bounded intrinsic diameter, and further prove it satisfies a uniform 1-content estimate. Combining these local controls inside the regular balls with the global covering theorem yields a uniform diameter bound for each time slice. Geometrically, the 1-content estimate asserts that, near any type of singularity, regions of high curvature occur in a controlled non-spiral and non-fractal fashion.

The covering argument in this direction originates from the highly influential work of Cheeger and Naber~\cite{CN13a}, who established a quantitative stratification theorem and obtained the quantitative estimates for manifolds with lower Ricci curvature bounds. To obtain the optimal estimates, the second named author and Naber~\cite{JN21} introduced the notion of neck regions, together with the techniques of neck analysis and neck decomposition. This framework has since proven remarkably powerful, leading to some major advances in problems involving manifolds with Ricci curvature bounds, Yang–Mills, harmonic maps \cite{NV19,CJN21,NV24} and many other related works \cite{Na18,LN20,Wa20,BNS22,FNS23,HJ23}.

The first application of the neck decomposition theory to the parabolic setting  was established in our previous work on linear parabolic partial differential equations \cite{HJ24}, where we obtained the optimal estimate for the measure of nodal set at each time-slice. More recently, Fang and Li~\cite{FL25} extended the neck decomposition approach to mean curvature flow under the mean-convexity condition, where all singularities are of cylindrical type. 

In this paper, we prove that, in the flow setting, the covering Theorem \ref{t:slice covering} can be refined purely in terms of \emph{regular balls}, without the need for neck regions in the final decomposition. Moreover, in dimension two, the argument can be further simplified by exploiting the entropy gap based on \cite{BL17} (also \cite{CIMW13, Bre16}), which allows us to avoid delicate iterative covering process. To prove the covering theorem, we establish a strengthened quantitative splitting result in the spirit of ~\cite{Wh97,CHN13}, based on the uniqueness of cylindrical singularity ~\cite{CM15,CM25} and the resolution of the multiplicity-one conjecture by Bamler-Kleiner \cite{BK23}. While previous applications of such coverings focused on obtaining measure estimates and proving rectifiability of the singular set, we observe that this framework can also be applied to address distance control problems, leading to the uniform intrinsic diameter bound established in this work. In addition, our approach extends naturally to other geometric flows, including the Ricci flow, where we obtain analogous results building upon Bamler's seminal work on the compactness theory of Ricci flows \cite{Ba20a,Ba20b,Ba20c}\footnote{After the completion of this work, an interesting related preprint \cite{Gi25} on the diameter bound for three-dimensional Type-I Ricci flows appeared on the arXiv. Our results were obtained independently.}.

In higher dimension $\RR^{n+1}$ with $n > 2$, under the mean convexity assumption, the rectifiability and finiteness of the singular set measure were established by Colding and Minicozzi \cite{CM16}, and later extended in quantitative form by Fang and Li \cite{FL25}. Without much effort, our method could be used to prove the sharp $L^1$ bound for $A$ on each slice and the weak $L^3$ bound for $A$ in the space-time region. However, in the absence of mean convexity, very little is known in higher dimensions. For example, the multiplicity one remains an open problem. We will discuss more about these issues in a forthcoming sequel.

\textbf{Acknowledgements.}  
The authors would like to thank  Professors Otis Chodosh, Ben Chow, Toby Colding,  Aaron Naber and  Felix Schulze for their interest in this work. The authors would like to thank Aaron Naber for several helpful suggestions on an earlier version of this paper. The first author thanks Tang-Kai Lee, Zhihan Wang, Xinrui Zhao and Jingze Zhu for some helpful discussions. 
W. Jiang was supported by National Key Research and Development Program of China (No. 2022YFA1005501), National Natural Science Foundation of China (Grant No. 12125105 and 12071425), the Jonathan M. Nelson Center for Collaborative Research and the Ky Fan and Yu-Fen Fan Endowment Fund. Y. Huang was partially supported by Simons Dissertation Fellowship in Mathematics.

\section{Preliminary}
In this section, we will review some results in mean curvature flow which will be used in our proofs and also unify the notations. 

\subsection{Notations}

For any space–time point $X \in \mathbb{R}^3 \times \mathbb{R}$, we write $X = (x, t_X)$, where $x$ denotes the spatial component and $t_X$ the time.  

Let $B_r(x)$ denote the Euclidean ball
\[
B_r(x) = \{ y \in \mathbb{R}^3 : ||x - y|| < r \}.
\]
We define the parabolic ball
\[
Q_r(X) = \{ (y,s) \in \mathbb{R}^3 \times \mathbb{R} : \|y - x\| \le r,\, |s - t_X| \le r^2 \}.
\]

Throughout the paper, $\|\cdot\|$ denotes the Euclidean norm on $\mathbb{R}^3$.  
The standard parabolic metric on $\mathbb{R}^3 \times \mathbb{R}$ is defined by
\[
d\big( (x,t), (y,s) \big) = \max \{ \|x - y\|,\, |t - s|^{1/2} \}.
\]
With respect to this metric, the parabolic ball $Q_r(x,t)$ is precisely the metric ball of radius $r$.  
We denote by $\mathcal{H}^k$ the $k$-dimensional Hausdorff measure in $\mathbb{R}^3$, and by $\mathcal{H}^k_P$ the $k$-dimensional Hausdorff measure associated to the parabolic metric in $\mathbb{R}^3 \times \mathbb{R}$.

For $\lambda > 0$, let
\[
D_{\lambda}(x,t) = (\lambda x, \lambda^2 t)
\]
denote the parabolic scaling.  
Given $X \in \mathbb{R}^3 \times \mathbb{R}$ and $r>0$, we define the \emph{rescaled flow} by
\begin{align}\label{e:rescalingnotation}
   \mathcal{M}_{X,r} = D_{r^{-1}} (\mathcal{M} - X), 
\end{align}
that is, we translate the point $X$ to the origin and rescale parabolically by the factor $r^{-1}$.

\subsection{Mean curvature flow and almost regular flow}

A one-parameter family of surfaces $\{ \mathcal{M}_t \}_{t \in I}$ in $\mathbb{R}^3$ evolves by \emph{mean curvature flow} (see \cite{Hu84}) if
\[
(\partial_t x)^{\perp} = \vec{H},
\]
where $\vec{H}$ denotes the mean curvature vector of $\mathcal{M}_t$.

Given a smooth, closed, embedded surface $\mathcal{M}_0 \subset \mathbb{R}^3$, there exists a unique smooth solution $\mathcal{M}_t$ to mean curvature flow on a maximal time interval $[0,T)$ satisfying
\[
\lim_{t \to T} \max_{\mathcal{M}_t} |A| = \infty.
\]
Beyond the first singular time, the flow can be continued in various weak senses. One natural extension is the notion of a \emph{Brakke flow}, introduced by Brakke~\cite{Bra78}. It provides a measure-theoretic formulation in terms of varifolds satisfying a certain inequality. We recall the standard definition (see Ilmanen~\cite{Il94}; Bamler–Kleiner~\cite{BK23}):

\begin{definition}[Brakke flow]
A two-dimensional Brakke flow on a time interval $I$ is a family of Radon measures $\{ \mu_t \}_{t \in I}$ on $\mathbb{R}^3$ satisfying:
\begin{enumerate}
    \item For almost every $t \in I$, the measure $\mu_t$ is integer $\mathcal{H}^2$-rectifiable, and the associated varifold has locally bounded first variation with generalized mean curvature vector $\vec{H} \in L^1_{\mathrm{loc}}(\mu_t)$.
    \item For every compact $K \subset \mathbb{R}^3$ and every $[t_1,t_2] \subset I$,
    \[
    \int_{t_1}^{t_2} \! \int_K |\vec{H}|^2 \, d\mu_t\, dt < \infty.
    \]
    \item For all compactly supported, nonnegative $u \in C_c^1(\mathbb{R}^3 \times [t_1,t_2])$,
    \[
    \int_{\mathbb{R}^3} u(\cdot,t)\, d\mu_t \Big|_{t=t_1}^{t=t_2}
    \le \int_{t_1}^{t_2} \! \int_{\mathbb{R}^3}
        (\partial_t u + \nabla u \cdot \vec{H} - u |\vec{H}|^2)
        \, d\mu_t\, dt.
    \]
\end{enumerate}
\end{definition}

A Brakke flow $\{ \mu_t \}_{t \in I}$ is said to be \emph{regular} on an open set $U \subset \mathbb{R}^3 \times I$ if there exists a smooth, properly embedded mean curvature flow $\mathcal{M} \subset U$ such that $\mu_t|_{U_t} = \mathcal{H}^2|_{\mathcal{M}_t}$ for all $t$.  
If $\mathcal{M}_t$ is a smooth mean curvature flow, then $\mu_t = \mathcal{H}^2|_{\mathcal{M}_t}$ defines a Brakke flow that is regular everywhere.

We write
\[
\mathcal{M} = \bigcup_{t \in I} \overline{ (\mathrm{supp}\,\mu_t) \times \{t\} }.
\]
A point $X \in \mathcal{M}$ is called \emph{regular} if the flow is regular in some parabolic neighborhood $Q_r(X)$, and the \emph{singular set} $\mathcal{S}$ is defined as the complement of the regular points.

\medskip

In $\mathbb{R}^3$, Bamler and Kleiner introduced the notion of an \emph{almost regular flow}—a class of well-behaved Brakke flows that are regular at almost every time and whose support is unit-regular (see definition in the next subsection) and multiplicity one.  
We do not reproduce the full definition here and instead record the main properties that we will use; see~\cite{BK23} for a complete exposition.

\begin{theorem}[Bamler–Kleiner \cite{BK23}] \label{t:BK}
Let $\mathcal{M} \subset \mathbb{R}^3 \times I$ be a bounded almost regular flow, and let
\[
\mathcal{M}^i = D_{r_i^{-1}}(\mathcal{M} - X_i)
\]
be a blow-up sequence with $r_i \to 0$ and uniformly bounded $X_i$. Then:
\begin{enumerate}
    \item After passing to a subsequence, $\mathcal{M}^i \to \mathcal{M}^\infty$ as Brakke flows, where $\mathcal{M}^\infty$ is itself an almost regular flow. Moreover, the convergence is locally smooth at every regular time.
    \item If $\mathcal{M}^\infty$ is stationary (i.e., a minimal surface), then $\mathcal{M}^\infty$ is a static plane.
\end{enumerate}
\end{theorem}

The class of almost regular flows includes most standard weak flows “through singularities”, such as the outer and inner mean curvature flows starting from any smooth, closed, embedded surface in $\mathbb{R}^3$. In particular, every smooth mean curvature flow and every non-fattening level set flow with compact time-slices is almost regular \cite{BK23}.

\subsection{Monotonicity and entropy}

Let $X_0 = (x_0, t_0) \in \mathbb{R}^3 \times I$.  
The backward heat kernel centered at $X_0$ is defined by
\[
    \phi_{X_0}(x,t) = \frac{1}{4\pi (t_0 - t)} 
    \exp\!\left(- \frac{|x - x_0|^2}{4 (t_0 - t)}\right),
    \qquad t < t_0.
\]
The following fundamental monotonicity result is due to Huisken~\cite{Hu90}: If $\mathcal{M}_t$ is a Brakke flow, then for all $X_0 = (x_0, t_0)$ we have
\begin{equation}\label{e:monotonicity}
    \overline{D}_t \int \phi_{X_0} \, d\mathcal{M}_t
    \le
    - \int 
    \Big| \vec{H} + \frac{(x - x_0)^{\perp}}{2(t_0 - t)} \Big|^2
    \phi_{X_0} \, d\mathcal{M}_t,
\end{equation}
where $\overline{D}_t$ denotes the upper Dini derivative.

For $r > 0$, we define the \emph{Gaussian density} of the flow at $X_0$ and scale $r$ by
\[
    \Theta_{X_0}(r)
    = \int \phi_{X_0}(x, t_0 - r^2)\, d\mathcal{M}_{t_0 - r^2}(x).
\]
By~\eqref{e:monotonicity}, $\Theta_{X_0}(r)$ is monotone non-decreasing as $r \downarrow 0$.  
The \emph{Gaussian density at $X_0$} is defined by
\[
    \Theta_{X_0} = \lim_{r \to 0^+} \Theta_{X_0}(r).
\]
A Brakke flow is said to be \emph{unit-regular} if it is regular at every point $X_0$ with $\Theta_{X_0} = 1$.

When $\Theta_{X_0}(r)$ is constant in $r$, equality in~\eqref{e:monotonicity} holds, and the flow is self-similar with respect to $X_0$.  
In this case, the time-slices satisfy the \emph{shrinker equation}
\[
    \vec{H} = \frac{(x - x_0)^{\perp}}{2(t - t_0)}.
\]

Following Colding–Minicozzi~\cite{CM12}, we define the \emph{entropy} of an integral $\mathcal{H}^2$-rectifiable Radon measure $\mu$ in $\mathbb{R}^3$ by
\[
    \lambda(\mu)
    = \sup_{x_0 \in \mathbb{R}^3,\, \tau > 0}
    \int \frac{1}{4\pi \tau}
    \exp\!\left( -\frac{|x - x_0|^2}{4\tau} \right)
    d\mu(x).
\]
For a Brakke flow $\mathcal{M}_{t \in I}$, Huisken’s monotonicity implies that
$t \mapsto \lambda(\mathcal{M}_t)$ is non-increasing.  
We define the \emph{entropy of the flow} by
\[
    \lambda(\mathcal{M}) = \sup_{t \in I} \lambda(\mathcal{M}_t).
\]
If the flow is generated by a self-shrinker, the entropy is constant and equals the entropy of that shrinker.

For shrinkers in $\mathbb{R}^3$, there is a definite entropy gap between self-shrinking cylinders and all other shrinkers in $\RR^3$, due to Bernstein-Wang~\cite{BL17} (also \cite{CIMW13, Bre16}: there exists a constant $\lambda_{\mathrm{gap}} > 0$ such that for any smooth, embedded self-shrinker
$\Sigma \subset \mathbb{R}^3$ that is not one of the standard models $\mathbb{R}^2$, $\mathbb{S}^2(\sqrt{2})$, or $\mathbb{S}^1(\sqrt{2}) \times \mathbb{R}$, we have
\begin{equation}\label{e:entropy gap}
    1 = \lambda(\mathbb{R}^2)
    < \lambda(\mathbb{S}^2(\sqrt{2}))
    < \lambda_1 \equiv \lambda(\mathbb{S}^1(\sqrt{2}) \times \mathbb{R})
    \le \lambda(\Sigma) - \lambda_{\mathrm{gap}}.
\end{equation}
In this paper, we take $\lambda_{gap}<< \lambda_1 -\lambda(\mathbb{S}^2(\sqrt{2}))  $. We will make use of this entropy gap in the proof of our covering theorems.

\subsection{Almost regular and cylindrical}

We record quantitative regularity and cylindricity notions and their consequences.

\begin{definition}
    We say $X$ is $(\delta,r)$-regular if $\sup_{Q_{\delta^{-1} r}(X) \cap \cM} r|A| \le \delta^2$.
\end{definition}

\begin{definition}[Almost Cylindrical]\label{d:almost cylindrical}
    Let $\cM \subset \RR^3 \times [0,T]$ be a bounded almost regular flow. We say that $X=(x,t_X)$ is \textbf{$(\delta,r)$-cylindrical} with respect to $\cL_X =  (x + L_X, t_X)$ if the following holds: 
    \begin{enumerate}
        \item for any $s \in [\delta r, \delta^{-1} r]$, the slice $\cM_{X,s} \cap (B_{\delta^{-2}}(0) \times\{t = -1\})$ can be written as a $C^{2,\alpha}$ graph of a function over a fixed round cylinder $\RR \times \mathds{S}^1(\sqrt{2})$ with $C^{2,\alpha}$-norm less than $\delta^2$. Recall that we have used the rescaling notation $\cM_{X,s}$ in \eqref{e:rescalingnotation};
        \item the axis of the round cylinder is a 1-dimensional linear subspace $L_X \subset \RR^3$. 
    \end{enumerate}
\end{definition}

First we prove that if $X$ is almost cylindrical with respect to $\cL_X$, then all the points near $\cL_X$ are almost cylindrical at the same scale, with respect to the same direction.

\begin{lemma}\label{l:near cylindrical also cylindrical}
    For any $\delta>0$, $\eta\le \eta_0(\delta)$, we have the following. Suppose $X$ is $(\eta,r)$-cylindrical with respect to $\cL_X = (x+L_X,t_X)$. Then any $Y \in Q_{\eta r}(\cL_X) \cap Q_{2r}(X)$ is $(\delta,r)$-cylindrical with respect to $(y+L_X, t_Y)$.  
\end{lemma}

\begin{proof}
    The proof follows directly from the Definition \ref{d:almost cylindrical}. Let $\delta>0$ be given. If $\eta$ is chosen small, then for any $s\in [\delta r, \delta^{-1}r]$ and any $Y \in Q_{\eta r}(\cL_X) \cap Q_{2r}(X)$, we have $\cM_{Y,s}\cap (B_{\delta^{-2}}\times \{t=-1\} ) \subset \cM_{X,s'}\cap (B_{\eta^{-2}} \times \{t=-1\})$ for some $s' \in [\eta r,\eta^{-1}r]$. Hence $\cM_{Y,s}\cap(B_{\delta^{-2}}\times \{t=-1\} ) $ can be written as a $C^{2,\alpha}$ graph of a function over the cylinder with the axis $L_X$, with $C^{2,\alpha}$-norm smaller than $2\eta$. This completes the proof.
\end{proof}

As a corollary, near $\cL_X$, the Gaussian density is close to $\lambda_1$, the entropy of the round cylinder.

\begin{lemma}\label{l:almost cylindrical implies almost pinched}
Let $\cM$ be a Brakke flow in $\RR^3$ with entropy bounded by $\Lambda_0$. For any $\delta>0$, suppose $X$ is $(\eta,r)$-cylindrical with respect to $\cL_X$ for some $\eta\le \eta_0(\Lambda_0,\delta)$. Then, for every 
\[
Y\in Q_{\eta r}(\cL_X)\cap Q_{2r}(X)\quad\text{and}\quad s\in[\delta r,\delta^{-1}r],
\]
we have
\[
\big|\Theta_Y(s)-\lambda_1\big|\le \delta .
\]
\end{lemma}

\begin{proof}
By Lemma~\ref{l:near cylindrical also cylindrical}, it suffices to prove the estimate at $X$; the same argument then works for all such $Y \in Q_{\eta r}(\cL_X) \cap Q_{2r}(X)$.

For any $s\in [\delta r, \delta^{-1}r]$, by the scaling property of $\Theta$ we have
    \begin{equation*}
        \Theta_X(s) = (4\pi)^{-1} (\int e^{-\frac{|y|^2}{4}} d\mu),
    \end{equation*}
    where $d\mu$ is the associated measure of the rescaled flow $\cM_{X,s}$ at time $t=-1$. 

    Since $ (\cM_{X,s}\cap \{t=-1\}) \cap B_{\eta^{-2}}(0)$ is a $C^{2,\alpha}$ graph over $\RR \times \mathds{S}^1(\sqrt{2})$ with graph norm smaller than $\eta^2$, if $\eta\le \eta(\Lambda_0,\delta)$ we have
    \begin{equation*}
       \Big|\int_{B_{\eta^{-2}}(0)} e^{-\frac{|y|^2}{4}} d\mu- \int_{B_{\eta^{-2}(0)}} e^{-\frac{|y|^2}{4}}d\mu_0 \Big|\le  \delta^2,
    \end{equation*}
where $d\mu_0$ is the standard measure for round cylinder. Next we estimate the part outside $B_{\eta^{-2}}$. First note that since the entropy of the flow is bounded by $\Lambda$, then by \cite{CM12} each $\cM_{X,s}$ has polynomial volume growth, i.e. $\mu(B_R(0)) \le C(\Lambda_0)R^2$ for any $R>1$. Then we use the exponential decay of the Gaussian to obtain
\begin{equation*}
\begin{split}
    \int_{\RR^3 \setminus B_{\eta^{-2}}(0)} e^{-\frac{|y|^2}{4}} d\mu &= \sum_{k=1}^{\infty} \int_{B_{(k+1)\eta^{-2}}(0) \setminus B_{k\eta^{-2}}(0)} e^{-\frac{|y|^2}{4}} d\mu \le \sum_{k=1}^{\infty} \mu( B_{(k+1)\eta^{-2}}(0) ) \cdot e^{-\frac{k^2}{4\eta^4}}  \\
    & \le C(\Lambda_0) \sum_{k=1}^{\infty} \frac{(k+1)^2}{\eta^4} \cdot e^{-\frac{k^2}{4\eta^4}} \le \delta^2,
\end{split}
\end{equation*}
provided $\eta \le \eta(\Lambda_0,\delta)$ small enough. Combining these two implies the density estimate.  
\end{proof}

\section{Entropy and Splitting}

In this section, we will prove the quantitative splitting theorem in mean curvature flow. The argument usually proceeds by contradiction, relying on suitable compactness theorem. For mean curvature flows in $\RR^3$, we observe that the Gaussian density determines the type of splitting due to the entropy gap \eqref{e:entropy gap}. In particular, when the density is close to either $1$ or $\lambda_1$, the flow must be quantitatively close to a plane or a round cylinder, respectively. This dichotomy implies that the geometry and spatial distribution of density–pinched points are effectively governed by the value of the density itself. This observation simplifies the proof of the covering theorem in the next section.

If a Brakke flow is self-similar with respect to more than one point, strong rigidity follows.  
We recall the standard cone-splitting principle (see \cite{Wh97, CHN13}).

\begin{lemma}[Cone-splitting for Brakke flows]\label{l:standard split}
Let $\mathcal{M}$ be a Brakke flow that is self-similar with respect to two distinct points
$X = (x, t_X)$ and $Y = (y, t_Y)$. Then:
\begin{enumerate}
    \item If $x \neq y$, then $\mathcal{M}$ is translation invariant along the direction $x - y$.
    \item If $t_X \neq t_Y$, then $\mathcal{M}$ is static for all $t \le \max\{t_X, t_Y\}$.
\end{enumerate}
\end{lemma}

In this section, we derive a quantitative cone-splitting theorem, by studying the $\lambda$-pinched set $\cV_{\eta,r}(X;\lambda) \subset Q_{2r}(X)$, at which the Gaussian density is close to $\lambda$ near the scale $r$ at $X$. 
\begin{definition}
    Consider $Q_{2r}(X)$ with $t_X \ge 4\eta^{-2}r^2$. We define the $\lambda$-pinched subset in $Q_{2r}(X)$ to be 
    \begin{equation}
        \cV_{\eta,r}(X;\lambda) \equiv \{ Y \in Q_{2r}(X) : \Theta_Y(\eta r) \ge \lambda - \eta^2 \ge \Theta_Y(\eta^{-1}r) - \eta^2 \}.
    \end{equation}
\end{definition}

First we see that if the $\lambda$-pinched points spread out enough, then $\lambda$ must be either close to $1$ or $\lambda_1$, and that $X$ is almost regular or almost cylindrical. Similar to standard cone-splitting lemma \ref{l:standard split}, we prove a quantitative cone-splitting theorem here.

\begin{theorem}\label{t:Quant split}
Let $\cM \subset \RR^3 \times I$ be a bounded almost regular flow. For any $\delta>0$, $\tau>0$, $r\le r_0(\cM,\delta,\tau)$ and $\eta\le \eta_0(\cM,\delta,\tau)$, the following holds: if for some $\lambda\ge 1$
\begin{enumerate}
    \item $\Theta_X(\eta r) \ge \lambda - \eta^2 \ge \Theta_X(\eta^{-1}r) - \eta^2 \text{ for some } X=(x,t_X) \text{ with } t_X \ge 4\eta^{-2}r^2$
\item $\Theta_Y(\eta r) \ge \lambda - \eta^2 \ge \Theta_Y(\eta^{-1}r) - \eta^2 \text{ for some } Y=(y,t_Y) \in Q_{r}(X)\setminus Q_{\tau r}(X)$,
\end{enumerate}
then exactly one of the following holds:
\begin{enumerate}
    \item $\lambda \in [1,1+ \frac{1}{2}\lambda_{gap}]$ and $X$ is $(\delta,r)$-regular.
    \item $\lambda \in [\lambda_1-\frac{1}{2}\lambda_{gap},\lambda_1+\frac{1}{2}\lambda_{gap}]$ and $X$ is $(\delta,r)$-cylindrical with respect to $(x+L_{y-x},t_X)$, where $L_{y-x}$ is the one-dimensional linear subspace spanned by $y-x$.
\end{enumerate}
\end{theorem}

\begin{proof}
We argue by contradiction. Suppose that the statement fails. Then there exist parameter $\tau_0>0$, and sequences
\[
r_i \to 0,\quad \eta_i \to 0,\quad \lambda_i \ge 1,\quad X_i=(x_i,t_i),\quad Y_i=(y_i,s_i)\in Q_{r_i}(X_i)\setminus Q_{\tau_0 r_i}(X_i)
\]
such that
\begin{enumerate}
    \item $\Theta_{X_i}(\eta_i r_i) \ge \lambda_i - \eta_i^2 \ge \Theta_{X_i}(\eta_i^{-1}r_i) - \eta_i^2$;
    \item $\Theta_{Y_i}(\eta_i r_i) \ge \lambda_i - \eta_i^2 \ge \Theta_{Y_i}(\eta_i^{-1}r_i) - \eta_i^2$;
\end{enumerate}
and
\[
\lambda_i \notin [1,1+\frac{1}{2}\lambda_{\mathrm{gap}}]\cup[\lambda_1-\frac{1}{2}\lambda_{\mathrm{gap}},\lambda_1+\frac{1}{2}\lambda_{\mathrm{gap}}].
\]

By Huisken’s monotonicity formula, each $\lambda_i$ is bounded above by $\lambda(\cM_0)$. Passing to a subsequence, we may assume that $\lambda_i \to \lambda_\infty \notin [1,1+\lambda_{\mathrm{gap}})\cup(\lambda_1-\lambda_{\mathrm{gap}},\lambda_1+\lambda_{\mathrm{gap}})$.

Consider the rescaled flows $\cM_i \equiv r_i^{-1}(\cM - X_i)$. By Brakke's compactness (see \cite{Bra78,Il94,BK23}), after passing to a subsequence we have
\[
\cM_i \to \cM_\infty
\]
as Brakke flows, with $X_i \to \mathbf{0}$ and $Y_i \to Y \in Q_1(\mathbf{0})\setminus Q_{\tau_0}(\mathbf{0})$. Moreover,
\[
\Theta_{\mathbf{0}}(s)=\lambda_\infty \quad\text{and}\quad \Theta_Y(s)=\lambda_\infty \quad\text{for all } s>0.
\]
By Theorem \ref{t:BK}, $\cM_\infty$ is the associated flow of an almost regular flow. In particular, it is regular at almost every time, and $\cM_i$ converges smoothly to $\cM_\infty$ at those times. Since $\Theta_{\mathbf{0}}(s)\equiv\lambda_\infty$, $\cM_\infty$ is self-similar with respect to $\mathbf{0}$. Combining these implies that $\cM_\infty$ is regular for any $t<0$ and is generated by a smooth shrinker in $\RR^3$ with entropy $\lambda_\infty$.  

The same reasoning applied at $Y$ gives that $\cM_\infty$ is also self-similar with respect to $Y=(y,t_Y)$.

\medskip\noindent
\textbf{Case 1: $|t_Y|\ge\tau_0^2$.}  
By the cone-splitting Lemma \ref{l:standard split}, $\cM_\infty$ must be static for all $t\in(-\infty,T]$, where $T\ge\max\{0,t_Y\}$. Consequently, $\cM_\infty\cap\{t\}$ is a minimal cone for all $t\le T$. By theorem \ref{t:BK}, the only possible such limit is the multiplicity-one plane. Combining with unit-regularity, this implies that $\cM_\infty$ is a static multiplicity-one plane, hence $\lambda_\infty=1$, contradicting our assumption on $\lambda_\infty$.

\medskip\noindent
\textbf{Case 2: $|t_Y|<\tau_0^2$.}  
Then necessarily $|y|\ge\tau_0$. By the cone-splitting Lemma \ref{l:standard split}, $\cM_\infty$ is invariant under translation in the direction $y$. Thus
\[
\cM_\infty = \RR \times N,
\]
where $N$ is a smooth curve shortening flow generated by a one-dimensional embedded shrinker for $t<0$. By the classification of one-dimensional smooth embedded shrinkers \cite{AL86}, $N$ is either a shrinking round circle or a straight line. Hence $\cM_\infty$ is either a shrinking round cylinder or a static plane, so $\lambda_\infty\in\{1,\lambda_1\}$, again a contradiction.

This contradiction proves that $\lambda$ must lie in $[1,1+\frac{1}{2}\lambda_{\mathrm{gap}}]\cup[\lambda_1-\frac{1}{2}\lambda_{\mathrm{gap}},\lambda_1+\frac{1}{2}\lambda_{\mathrm{gap}}]$.

To establish the second part, suppose by contradiction that $1\le\lambda_i\le1+\frac{1}{2}\lambda_{\mathrm{gap}}$ but each $X_i$ fails to be $(\delta_0,r_i)$-regular. Then the same compactness argument yields $\cM_\infty$ with $\lambda(\cM_\infty)=1$, so $\cM_\infty$ is the static plane. Smooth convergence then implies $X_i$ is $(\delta_0,r_i)$-regular for $i$ large, a contradiction.

Similarly, if $|\lambda_i-\lambda_1|\le \frac{1}{2}\lambda_{\mathrm{gap}}$ while $X_i$ is not $(\delta_0,r_i)$-cylindrical, the limit satisfies $\lambda(\cM_\infty)=\lambda_1$, hence $\cM_\infty$ is the shrinking round cylinder. Smooth convergence again implies that $X_i$ must be $(\delta_0,r_i)$-cylindrical for large $i$, yielding a contradiction. The theorem follows.
\end{proof}

As a direct corollary of Theorem \ref{t:Quant split}, we conclude that the set of $\lambda$-pinched points with $\lambda$ far away from $1$ and $\lambda_1$ is contained in some small ball.

\begin{lemma}\label{l:pinched points near axis}
Let $\cM \subset \RR^3 \times I$ be a bounded almost regular flow. For any $\tau>0$, $r\le r_0(\cM,\tau)$ and $\eta\le \eta_0(\cM,\tau)$, the following holds for any $X_0$ with $t_{X_0} \ge 4\eta^{-2}r^2$. Suppose $ X\in \cV_{\eta,r}(X_0;\lambda)$ for some $\lambda \notin [1,1+\frac{1}{2}\lambda_{\mathrm{gap}})\cup(\lambda_1-\frac{1}{2}\lambda_{\mathrm{gap}},\lambda_1+\frac{1}{2}\lambda_{\mathrm{gap}})$, then
\begin{equation*}
    \cV_{\eta,r}(X_0;\lambda) \subset Q_{\tau r}(X)\,.
\end{equation*}
\end{lemma}

\begin{proof}
Suppose that there exists $ Y\in \cV_{\eta,r}(X_0;\lambda) \setminus Q_{\tau r}(X)$. Then Theorem~\ref{t:Quant split} implies that $\lambda \in [1,1+\frac{1}{2}\lambda_{\mathrm{gap}}]\cup[\lambda_1-\frac{1}{2}\lambda_{\mathrm{gap}},\lambda_1+ \frac{1}{2}\lambda_{\mathrm{gap}}]$, contrary to the hypothesis. Hence no such $Y$ exists. This completes the proof.
\end{proof}

Next we study the case where $\lambda$ is close to $\lambda_1$. If we only assume that the density is pinched at $\lambda$ at only one point $X$, but the density $\lambda$ is close to $\lambda_1$, then we can still conclude that $X$ is almost cylindrical. The proof uses the same contradiction argument.

\begin{proposition}\label{p:almost cylindrical rigidity}
    Let $\cM \subset \RR^3 \times I$ be a bounded almost regular flow. For any $\delta>0$, $r\le r_0(\cM,\delta)$ and $\eta\le \eta_0(\cM,\delta)$, the following holds for any $X_0$ with $t_{X_0} \ge 4\eta^{-2}r^2$. Suppose $ X = (x,t_X)\in \cV_{\eta,r}(X_0;\lambda)$ for some $\lambda \in [\lambda_1-\lambda_{gap},\lambda_1 +  \lambda_{gap}]$. Then we have the following
    \begin{enumerate}
        \item $X$ is $(\delta,s)$-cylindrical with respect to $\cL_{X} =(x+L_X,t_X)$ for any $s\in [\delta^2 r,r]$. Equivalently, in the normalized flow $\cM_{X,s'}$ for any $s'\in[\delta^3 r, \delta^{-1}r]$, the time-slice $\{t=-1\}$ inside the ball $B_{\delta^{-2}}(0)$ is a $C^{2,\alpha}$-graph with norm $\le\delta^2$ over a round cylinder with axis $L_X$.
        \item The $\lambda$-pinched set $\cV_{\eta,r}(X_0;\lambda)$ lies near the axis $\cL_{X}$:
        \begin{equation*}
            \cV_{\eta,r}(X_0;\lambda)\subset Q_{\delta r}(\cL_X)\cap Q_{2r}(X_0)\,.
        \end{equation*}
        \item For any $Y \in \cV_{\eta,r}(X;\lambda)$ that is $(\delta,r)$-cylindrical with respect to $\cL_Y=(y+L_Y,t_Y)$, we have
        \begin{equation*}
            d_H(\cL_X \cap Q_{2r}(X_0) , \cL_Y \cap Q_{2r}(X_0)) \le \delta^2 r ~~ \text{ and } ~~ d_H(L_X\cap B_1(0), L_Y \cap B_1(0)) \le \delta^2.
        \end{equation*}
    \end{enumerate}
\end{proposition}

\begin{proof}
    We argue by contradiction, similar to Theorem \ref{t:Quant split}. Suppose there exist sequences $r_i \to 0$, $\eta_i \to 0$, $X_{0;i}$ with $t_{X_{0;i}} \ge 4\eta^{-2}_i r_i^2$, $\lambda_i \in [\lambda_1 -\lambda_{gap}, \lambda_1+\lambda_{gap}]$  and $X_i \in \cV_{\eta_i,r_i}(X_{0;i};\lambda_i)$ while the conclude does not hold for some $\delta_0>0$. By passing to a subsequence, we may assume $\lambda_i \to \lambda_{\infty} \in [\lambda_1-\lambda_{gap},\lambda_1 +  \lambda_{gap}]$. 
    
    Consider the blow-ups $\cM_i := r_i^{-1}(\cM-X_i)$. By Theorem \ref{t:BK}, after passing to a subsequence $\cM_i \to \cM_\infty$ as Brakke flows, with smooth local convergence at almost every time. Huisken's monotonicity and the pinching imply that the Gaussian density at the origin of $\cM_\infty$ is constant at all scales and equals $\lambda_\infty$: $\Theta_{\mathbf 0}(s)\equiv \lambda_\infty$ for all $s>0$. Hence $\cM_\infty$ is a self-similar shrinking flow generated by a smooth shrinker $\Sigma_\infty\subset\RR^3$ with entropy $\lambda(\Sigma_\infty)=\lambda_\infty$. By the gap estimate \eqref{e:entropy gap}, $\lambda_{\infty} = \lambda_1$ and $\cM_{\infty}$ is generated by the round cylinder with axis $L_{\infty}$. Again, the smooth convergence  implies that $X_i$ is $(\delta_0, s)$-cylindrical for any $s\in [\delta_0^2r_i, r_i]$ for all large $i$. Contradiction arises. This proves (1). 

    For (2), suppose there exists a sequence of point $Y_i \in \cV_{\eta_i,r_i}(X_{0,i};\lambda_i)$ but $Y_i \notin Q_{\delta_0 r_i}(\cL_{X_i})\cap Q_{2r_i}(X_{0;i})$. By passing to a subsequence we assume $Y_i \to Y$ and thus $\Theta_Y(s) \equiv \lambda_{\infty} = \lambda_1$. Hence $Y \in \cL_{\infty}$ which contradicts that $d(Y_i,\cL_{X_i}) > \delta_0 r_i$ and $\cL_{X_i} \to \cL_{\infty}$. (3) is proved similarly, observing that for any $Y_i \in \cV_{\eta_i,r_i}(X_{0,i};\lambda_i)$, the line $\cL_{Y_i} \to \cL_{\infty}$. Therefore, this completes the proof.
\end{proof}

Before proceeding, we recall that once a flow becomes sufficiently close to the cylinder at multi-scales, it should remain close to the same cylinder at all scales. This persistence of cylindrical structure is a quantitative version of the uniqueness of tangent flows, and follows from the effective uniqueness theorem of Colding--Minicozzi \cite{CM25} for the rescaled mean curvature flow. The next proposition records this propagation property in the setting of bounded almost regular flows.

\begin{proposition}\label{p:propagation-cylinder}
Let $\cM\subset\RR^3\times I$ be a bounded almost regular flow. Let $\delta>0$ and $0<r_1<r_2\le r_0(\cM,\delta)$. Then there exists $\eta_0=\eta_0(\cM,\delta)>0$ such that for any $0<\eta\le\eta_0$ the following holds. Suppose $X_0$ satisfies $t_{X_0}\ge4\eta^{-2}r_2^2$, and that there exists $X\in Q_{2r_2}(X_0)$ and some
\[
\lambda\in[\lambda_1-\lambda_{\mathrm{gap}},\lambda_1+\lambda_{\mathrm{gap}}]
\]
with
\[
\Theta_X(\eta r_1)\ \ge\ \lambda-\eta^2\ \ge\ \Theta_X(\eta^{-1}r_2)-\eta^2.
\]
Then $X$ is $(\delta,s)$-cylindrical with respect to some fixed axis $\cL_X$ for every $s\in[r_1,r_2]$.
\end{proposition}

\begin{proof}
By Proposition \ref{p:almost cylindrical rigidity}, the pinching condition and smallness of $\eta$ imply that $X$ is $(\delta_*,r_2)$-cylindrical for some small $\delta_*=\delta_*(\cM)$. By Definition \ref{d:almost cylindrical}, for every $s\in[\delta_*r_2,\delta_*^{-1}r_2]$ the normalized slice $ \cM_{X,s}\cap (B_{\delta_*^{-2}} \times \{t=-1\})$ is a $C^{2,\alpha}$-graph with norm $\le\delta_*^2$ over the corresponding round cylinder denoted by $\cC_1$. By Lemma \ref{l:almost cylindrical implies almost pinched}, we have $|\Theta_X(s) - \lambda_1 | \le \delta_*$  for every $s\in[\delta_*r_2,\delta_*^{-1}r_2]$.

Set $\sigma=-\log(s/r_2)$ so that $s=r_2e^{-\sigma}$. Define the rescaled flow $\Sigma_{\sigma}:=\cM_{X,\,r_2e^{-\sigma}}\cap\{t=-1\}$, which satisfies the rescaled mean curvature flow equation $\partial_{\sigma} x=(\tfrac12x^\perp-H)$. Due to the rescaling and the pinching for $\Theta_X$, we conclude that $\big|(4\pi)^{-1}\int_{\Sigma_{\sigma}} e^{-\frac{|x|^2}{4}} - \lambda_1 \big| \le \delta_*$ for any $\sigma\in [\log \delta_*,-\log\delta_*-\log(r_1/r_2)]$. Therefore, we can apply \cite[Theorem 0.5]{CM25} to conclude that $B_{\delta_*^{-2}}(0)\cap\Sigma_{\sigma}$ can be written as a $C^{2,\alpha}$ graph over the same cylinder $\cC_1$ with graph norm smaller than $c(\delta_*)$ with $c(\delta_*) \to 0$ as $\delta_* \to 0$, for any $\sigma \in [1+\log\delta_*, -\log\delta_*-\log(r_1/r_2)]$. Reverting to the unrescaled variables, $X$ is $(\delta,s)$-cylindrical with respect to the fixed axis $\cL_X$ for every $s\in[r_1,r_2]$.
\end{proof}

Lastly we prove the almost rigidity for $\lambda$ close to $1$. See also Brakke's regularity theorem \cite{Bra78,Wh05,KT14,DPGS24}.

\begin{proposition}\label{p:almost regular rigidity}
    Let $\cM \subset \RR^3 \times I$ be a bounded almost regular flow. For any $\delta>0$, $\tau>0$, $r\le r_0(\cM,\delta,\tau)$  and $\eta\le \eta_0(\cM,\delta,\tau)$, the following holds for any $X_0$ with $t_{X_0} \ge 4\eta^{-2}r^2$. Suppose $ X = (x,t_X)\in \cV_{\eta,r}(X_0;\lambda)$ for some $\lambda \in [1, 1+  \lambda_{gap}]$. Then $X$ is $(\delta,s)$-regular. 
\end{proposition}

\begin{proof}
    We apply the same contradiction argument as in Proposition \ref{p:almost cylindrical rigidity}. The only difference is that the limit flow has entropy $1$ and hence is a static plane. Then smooth convergence then implies the result. 
\end{proof}

\section{Cylindrical Regions and bi-Lipschitz estimates}

We now introduce the cylindrical regions that will appear in the covering theorems. The idea first appears in Jiang-Naber \cite{JN21} for Einstein manifolds, has been generalized to many other settings \cite{NV19,CJN21,BNS22,NV24}.

\begin{definition}\label{d:cylindricalregion}
Let $\cM \subset \RR^3 \times I$ be a bounded almost regular flow. Let $\mathcal{C} \subset Q_{2r}(X_0)$ be a closed subset, and let $r_X : \mathcal{C} \rightarrow \RR_{\ge0}$ be a radius function and $\tau \le 10^{-10}$. The set $\mathcal{N} = Q_{2r}(X_0) \setminus \Bar{Q}_{r_X}(\cC)$ is called a $(\delta, r)$-cylindrical region if for any $X\in \cC$ the following hold:
\begin{enumerate}
    \item $\{Q_{\tau^2 r_X}(X) \}$ are pairwise disjoint for $X\in \cC$;
    \item for each $r_X \leq s \leq r $, $X$ is $
    (\delta,s)$-cylindrical with respect to $\mathcal{L}_{X} = L_X \times \{t_X\}$;
    \item for each $s \geq r_X$ with $Q_{2s}(X)\subset Q_{2r}(X_0)$, we have $\mathcal{L}_{X} \cap Q_s(X) \subset Q_{\tau s}(\mathcal{C})$ and $\mathcal{C} \cap Q_s(X) \subset Q_{\tau s}(\mathcal{L}_{X})$.
\end{enumerate}
Here $\Bar{Q}_{r_X}(\cC) \equiv \cup_{X\in \cC}\bar Q_{r_X}(X)$ is the union of the closure of central balls. 
\end{definition}

\begin{remark}
    In the definition, we can choose the axis $\mathcal{L}_{X}$ of the $(\delta,s)$-cylindrical $X$ depending only on the point $X$ but not on the scale $s$. This follows from the uniqueness of tangent flows of cylindrical singularities by Colding-Minicozzi \cite{CM15}, as recorded in Proposition \ref{p:propagation-cylinder}. 
\end{remark}

The key geometric feature is that the set of centers $\cC$ can be projected to the axis $\cL_X$ for $X \in \cC$ in a bi-Lipschitz way. 

\begin{lemma}\label{l:cylindrical region structure}
    For any $X\in \cC$, the projection map $\pi_X : \cC \to \cL_{X}$ is a bi-Lipschitz map, i.e. 
\begin{equation}\label{e:bilipschitz of centers}
    (1-2\tau)d(X_1,X_2) \le    ||\pi_X(X_1) - \pi_X(X_2)|| \le d(X_1,X_2), \text{ for any } X_1, X_2 \in \cC.
\end{equation}
\end{lemma}

\begin{proof}

Fix some $X$ and consider the projection map $\pi_X$. Take any $X_1, X_2 \in \cC$. For simplicity we assume $d(X,X_1) \ge d(X,X_2)$. Since $\pi_X$ is a projection, the right hand side of \eqref{e:bilipschitz of centers} holds trivially. Since $Q_{\tau^2 r_{X_1}}(X_1) \cap Q_{\tau^2 r_{X_2}}(X_2) = \emptyset$, we have $d\equiv d(X_1,X_2)\ge \tau^2 \max\{r_{X_1},r_{X_2}\}.$ By definition, $X_1$ is $(\delta,\tau^{-2}d)$-cylindrical with respect to $\cL_{X_1}=(x_1+L_{X_1}, t_{X_1})$ and $X_2$ is $(\delta,\tau^{-2}d)$-cylindrical with respect to $\cL_{X_2}= (x_2+ L_{X_2}, t_{X_2})$. By Proposition \ref{p:almost cylindrical rigidity}, if $\delta$ is chosen small, then we have $d_H(L_{X_1}\cap B_1(0), L_{X_2}\cap B_1(0)) \le \tau^2$ and
\begin{equation}\label{e:neck structure bilip 1}
    ||\pi_{X_1}(X_1 - X_2) || \ge (1-\tau) d
\end{equation}

Similarly, by Proposition \ref{p:almost cylindrical rigidity} we have $d_H(L_{X}\cap B_1(0), L_{X_1}\cap B_1(0)) \le \tau^2$, and thus 
\begin{equation}\label{e:neck structure bilip 2}
    ||\pi_X(X_1-X_2)   - \pi_{X_1}(X_1-X_2) || \le \tau^2 d(X_1,X_2) = \tau^2 d
\end{equation}

Combining \eqref{e:neck structure bilip 1} and \eqref{e:neck structure bilip 2} we conclude that
\begin{equation*}
    ||\pi_X(X_1-X_2)|| \ge ||\pi_{X_1}(X_1-X_2)|| - ||\pi_X(X_1-X_2) - \pi_{X_1}(X_1-X_2)|| \ge 
    (1-2\tau)d.
\end{equation*}
Hence this completes the proof of \eqref{e:bilipschitz of centers}.
\end{proof}

As an immediate consequence we obtain a uniform 1-content bound for the centers inside a cylindrical
region. Denote $\cC_{+}$ to be the subset of $\cC$ such that $r_X>0$ and $\cC_0=\cC\setminus \cC_{+}$. 

\begin{corollary}\label{c:measure upper bound}
Let $\cN \subset Q_{2r}(X_0) \setminus \Bar{Q}_{r_X}(\cC)$ be a $(\delta,r)$-cylindrical region. Then there exists some $C(\tau)$ such that 
    $\sum_{X\in \cC_+} r_X + \cH^1_P(\cC_0)  \le C(\tau) r$. 
\end{corollary}

\begin{proof}
    Let $X \in \cC$ and consider the projection map $\pi_X$. Since $\{Q_{\tau^2 r_Y}(Y),Y\in \cC_+\}$ is pairwise disjoint, then by \eqref{e:bilipschitz of centers} we have $B_{\tau^3 r_{X_i}}(\pi_X(X_i))$ is pairwise disjoint in $\cL_{X}$. This implies that 
\begin{equation*}
\begin{split}
     &\quad \sum_{X_i \in \cC_+} r_{X_i} + \cH^1(\cC_0 \cap Q_s(X))\\ 
     &\le C(\tau) \left(\sum_{X_i \in \cC_+ \cap Q_{s}(X)} \cH^1(B_{\tau^3 r_{X_i}}(\pi_X(X_i)) \cap \cL_{X}) + \cH^1(\pi_X(\cC_0 \cap Q_s(X)) \right)\\
    &\le C(\tau) r.
\end{split}
\end{equation*}
This completes the proof.
\end{proof}

Using the bi-Lipschitz structure of the centers established in Lemma~\ref{l:cylindrical region structure}, we can further derive a quantitative $\mu_{\cM}$-volume estimate for the neighborhood of the center set $\cC$ in space–time which will be used in the proof of weak $L^3$ for second fundamental form $A$.

\begin{lemma}\label{l:volume of center}
    Let $\cN \subset Q_{2r}(X_0) \setminus \Bar{Q}_{r_X}(\cC)$ be a $(\delta,r)$-cylindrical region. Then there exists some $C(\cM)>0$ such that for any $0<s<r$
    \begin{equation*}
        \mu_{\cM}(Q_{s}(\cC)) \le C(\cM) s^3r.
    \end{equation*}
\end{lemma}

\begin{proof}
    Let $s\in (0,r)$. Fix some $X\in \cC$ and the projection map $\pi_X : \cC \to \cL_X$. Let $\{X_i\}_{i \in I}$ be a maximal $s$-separated subset of $\cC$ with respect to the parabolic distance. Hence the balls $\{Q_{s/2}(X_i)\}_{i \in I}$ is pairwise disjoint. And we have the covering
    \begin{equation*}
        Q_{s}(\cC) \subset \bigcup_{i \in I} Q_{3s}(X_i).
    \end{equation*}

    By Lemma \ref{l:cylindrical region structure}, we have $\{Q_{s/4}(\pi_X(X_i))\}_{i \in I}$ is pairwise disjoint. Since $\cL_X$ is a 1-dimensional affine space, this implies that the cardinality $|I| \le 10rs^{-1}$. Therefore, combining the covering for $Q_s(\cC)$ we have
    \begin{equation*}
        \mu_{\cM}(Q_s(\cC)) \le \sum_{i \in I} \mu_{\cM}(Q_{3s}(X_i)) \le (10rs^{-1}) \int_{t_{X_i}-s^2}^{t_{X_i}+s^2} \mu_{\cM_t}(B_{3s}(x)) dt \le C(\cM) s^3 r,
    \end{equation*}
    where we use the local area estimate $\mu_{\cM_t}(B_s(x)) \le C(\cM)s^2$. 
    This completes the proof.
\end{proof}

\section{Main Covering Theorems}

In this section, we prove two main covering theorems. The first one is a covering for each time slice $\cM_{t_0}$, using almost regular balls with uniform $1$-content estimate. This new estimate is crucial for the proof of the diameter and curvature bounds.

\begin{theorem}[Time-slice Covering Theorem]\label{t:slice covering}
Let $\cM \subset \RR^3 \times I$ be a bounded almost regular mean curvature flow. For any $\delta \le \delta_0$, $R \le r_0(\cM,\delta)$,  the following holds. Let $(x,t_0) \in \cM_{t_0}$ with $t_0 \ge 4\eta^{-2}R^2$ for some $\eta(\cM,\delta)$. Then we have the following covering
    \begin{equation}
        B_R(x)\cap \cM_{t_0} \subset \bigcup_{i} B_{r_i}(x_i) \cup \bigcup_{j} \cC_{0,j;t_0} \cup S_{0;t_0}
    \end{equation}
where 
\begin{enumerate}
    \item each $(x_i,t_0)$ is $(\delta,r_i)$-regular, in particular we have $\sup_{B_{\delta^{-1}r_i}(x_i)\cap \cM_{t_0}} r_i|A| \le \delta^2$;
    \item the singular set $\cS_{t_0} \cap B_R(x) = (S_{0;t_0} \cup  \bigcup_j \cC_{0,j;t_0})\cap B_R(x)$. 
    \item $S_{0;t_0} \subset \cS^0$ is countable where $\cS^0$ is the $0$-stratum of the singular set. 
    \item each $\cC_{0,j;t_0}$ is contained in an embedded Lipschitz submanifold of dimension one. 
\end{enumerate}

Moreover, we have the following $1$-content estimate
\begin{equation}\label{e:slice 1 content}
    \sum_{i } r_i \le C(\cM,\delta,\eta) R.
\end{equation}
\end{theorem}

\begin{remark}
    If $\cM_{t_0}$ is smooth, then we have $B_R(x) \cap \cM_{t_0} \subset \bigcup_i B_{r_i}(x_i)$ with uniform $1$-content estimate for $r_i$, where each $(x_i,t_0)$ is $(\delta,r_i)$-regular.  The Lipschitz regularity in (4) could be improved to $C^{2,\alpha}$-regularity by \cite{SWX25}.
\end{remark}

The second one is a space-time covering theorem. We cover the space-time region of $\cM$ using regular balls and cylindrical regions, plus the singular part, with uniform $1$-content estimate. This is crucial for the proof of the estimate of space-time singular set and the curvature estimate in space-time regions. Some idea is similar with Jiang-Naber \cite{JN21}, Cheeger-Jiang-Naber \cite{CJN21} for elliptic cases and our previous work \cite{HJ24} for parabolic case.

\begin{theorem}[Space-time Covering Theorem]\label{t:spacetime cover}
Let $\cM \subset \RR^3 \times I$ be a bounded almost regular mean curvature flow. For any $\delta \le \delta_0$, $\tau<10^{-10}$, $R \le r_0(\cM,\delta,\tau)$ and $\eta \le \eta(\cM,\delta,\tau)$, the following holds. Let $X_0 \in \cM$ with $t_{X_0} \ge 4\eta^{-2}R^2$. Then we have the following covering
\begin{equation}
    Q_R(X_0) \cap \cM \subset \bigcup_a \big((\cN_a \cup \cC_{0,a})\cap Q_{r_a}(X_a) \big) \cup \bigcup_b Q_{r_b}(X_b) \cup S_0 
\end{equation}
such that
\begin{enumerate}
    \item $\cN_a \subset Q_{2r_a}(X_a) $ is a $(\delta,r_a)$-cylindrical region;
    \item $Q_{r_b}(X_b)$ satisfies that $X_b$ is $(\delta,r_b)$-regular;
    \item $\sum_a r_a + \sum_b r_b + \cH^1_P(S_0 \cup \bigcup_a\cC_{0,a}) \le C(\cM,\delta,\epsilon,\eta)R$;
    \item The singular set $\cS\cap Q_R(X_0) = (S_0 \cup \bigcup_a\cC_{0,a})\cap Q_R(X_0)$ is $1$-rectifiable. Moreover, $S_0$ is countable and each $\cC_{0,a}$ is contained in a embedded Lipschitz submanifold of dimension 1.
\end{enumerate}
\end{theorem}

We will first prove the second covering theorem, the space-time one, and then use it to prove the time-slice covering theorem. We consider two types of covering lemma. The first one is the covering for cylindrical balls.

\begin{lemma}[Covering of cylindrical balls]\label{l:c cover}
 For any $\delta \le \delta_0$, $\tau<10^{-10}$, $r \le r_0(\cM,\delta,\tau)$ and $\eta \le \eta(\cM,\delta,\tau)$, the following holds. Let $X_0$ with $t_{X_0}\ge 4\eta^2 r^{-2}$. Assume that $\sup_{Y\in Q_{2r}(X_0)}\Theta_Y(\eta^{-1}r) \le \lambda$ and that there exists some $X \in \cV_{\eta,r}(X_0;\lambda)$ for some $\lambda \in [\lambda_1 -\lambda_{gap}, \lambda_1 + \lambda_{gap}]$. Then we have the following decomposition
    \begin{equation*}
        Q_r(X_0) \subset (\cC_0 \cup \cN)  \cup \bigcup_{e \in E} Q_{r_e}(X_e),
    \end{equation*}
    where $\cN \subset Q_{2r}(X_0)$ is a $(\delta,r)$-cylindrical region and that any $Y \in Q_{2r_e}(X_e)$ satisfies that $\Theta_Y(\eta^{-1} r_e) < \lambda - \eta^2$. Moreover, $\cC_0$ is contained in an embedded Lipschitz submanifold of dimension 1 and there exists some constant $C(\tau,\eta)$ such that
    \begin{equation*}
        \sum_{e\in E} r_e +  \cH_P^1(\cC_0) \le C(\tau,\eta) r.
    \end{equation*}
\end{lemma}

\begin{proof}
    Fix some $\lambda \in [\lambda_1 -\lambda_{gap}, \lambda_1 + \lambda_{gap}]$. Since $X \in \cV_{\eta,r}(X_0;\lambda)$, by Proposition \ref{p:almost cylindrical rigidity}, $X$ is $(\eta',s)$-cylindrical with respect to $\cL_{X}= (x+L_X, t_X)$ for any $s \in [\tau^2 r,r]$, provided $\eta\le \eta(\eta')$ small enough. By Lemma \ref{l:near cylindrical also cylindrical}, any $Y \in Q_{\eta' s}(\cL_{X}) \cap Q_{2s}(X)$ is $(\delta, s)$-cylindrical with respect to $\cL_Y = (y+L_X, t_Y)$ for any $s \in [\tau^2 r, r]$, provided $\eta'\le \eta'(\delta,\tau)$ small enough.
    
    Consider a covering of $Q_{\eta' r}(\cL_X) \cap Q_{2r}(X)$ by balls $\{Q_{\tau r_1/10}(Y^1_i)\}_{i \in I_1}$ with $Y^1_i \in \cL_X$, $r_1=\tau r$ and $Q_{4\tau^2 r_1}(Y^1_i) \cap Q_{ 4\tau^2 r_1}(Y_j) = \emptyset$ whenever $Y^1_i \neq Y^1_j$. We define $\cC^1 = \cup_{i\in I_1} Y^1_i$ and then it is easy to see that $\cN^1 = Q_{2r}(X) \setminus \Bar{Q}_{r_1}(\cC^1)$ is a $(\delta,r)$-cylindrical region, provided $\eta$ small enough and the third condition in Definition \ref{d:cylindricalregion} satisfying with $\tau'\le \tau/2$. This is our first step covering. 
    
    Now we examine each ball $Q_{r_1}(Y^1_i)$. We classify the balls 
    \begin{equation*}
    \begin{split}
        E^1 &= \{ Y^1_i : \Theta_Y(\eta r_1) < \lambda - \eta^2 \text{ for any } Y \in Q_{2r_1}(Y^1_i) \}. \\
        F^1 &= \{Y^1_i\} \setminus E^1.
    \end{split}
    \end{equation*}
    
Then the ball $Q_{r_1}(Y_i^1)$ with center in $E^1$ has definite entropy drop for every point in $Q_{2r_1}(Y_i^1)$. We will see that these balls will be covered by $e$-ball as in the statement of the lemma. So we will not need to handle such balls until the final step. 
For each $Y^1_i \in F^1$, there exists some $X^2 \in Q_{r_1}(Y_i)$ such that $\Theta_{X^2}(\eta r_1) \ge \lambda -\eta^2$. By Proposition \ref{p:propagation-cylinder}, we have $X^2$ is $(\eta',s)$-cylindrical with respect to the fixed $\cL_{X^2} = (x^2 + L_{X^2}, t_{X^2})$ for any $s \in [\tau^2 r_1, r]$. By Lemma \ref{l:near cylindrical also cylindrical} and Proposition \ref{p:propagation-cylinder}, any $Y \in Q_{\eta' s}(\cL_{X^2})\cap Q_{2s}(Y^1_i)$ is $(\delta,s)$-cylindrical with respect to $\cL_Y = (y + L_{X^2}, t_Y)$ for any $s\in[\tau^2 r_1, r]$.

    Consider the covering of $Q_{\eta' r_1}(\cL_{X^2}) \cap Q_{2r_1}(Y^1_i)$ by balls $\{Q_{\tau r_2/10}(Y^2_i)\}$ with $Y^2_i \in \cL_{X^2}$, $r_2 = \tau^2 r$ and $Q_{4\tau^2 r_2}(Y^2_i) \cap Q_{4\tau^2 r_2}(Y^2_j) =\emptyset$ whenever $Y_i^2 \neq Y^2_j$. We do this covering for every $Y^1_i \in F^1$. Then we can collect all the centers and define $\cC^2 = E^1 \cup \bigcup_i \{Y^2_i\}$. We can check that $\cN^2 = Q_{2r}(X_0) \setminus \big(\bigcup_{e \in E^1}Q_{r_e}(X_e) \cup \bigcup_i \Bar{Q}_{r_2}(Y^2_i) \big)$ is a $(\delta,r)$-cylindrical region. Actually, the first two properties are satisfied by construction. We prove the third one for each $Y^2_i$ as the points in $E^1$ are checked in the last step. Note that $r_{Y^2_i} = r_2$. Since $\cL_{Y^2_i} = \cL_{X^2}$ inside the ball $Q_{2r_1}(Y^1_i)$, then we have 
$\text{ $\cL_{Y^2_i}\cap Q_s(Y^2_i) \subset Q_{\tau s/5}(\cC^2)$ and $\cC^2 \cap Q_s(Y^2_i) \subset Q_{\tau s/5}(\cL_{Y^2_i})$ }$
    for any $s\in [r_2,r_1]$. And for $s\in[r_1,r]$, by Proposition \ref{p:almost cylindrical rigidity}, we have $d_H(\cL_{Y_i^2}\cap Q_{s}(Y^2_i), \cL_{Y}\cap Q_s(Y^2_i)) \le \eta' s$ for any $Y \in \cC^2$ if $\eta\le \eta(\eta')$. We can pick $Y \in \cC^1$ and then the third property is checked using the third property in the cylindrical region $\cN^1$. Therefore, it follows that $\cN^2$ is a $(\delta, r)$-cylindrical region with 
\begin{align}\label{e:bettercovering5.8}
    \text{ $\cL_{Y}\cap Q_s(Y) \subset Q_{\tau s/5}(\cC^2)$ and $\cC^2 \cap Q_s(Y) \subset Q_{\tau s/5}(\cL_{Y})$ }
    \end{align}
    for any $Y\in \cC^2$ and $r\ge s\ge r_Y$. 

    Next we can classify the balls $Q_{r_2}(Y^2_i)$ and define
    \begin{equation*}
    \begin{split}
        E^2 &= E^1 \cup \{Y^2_i : \Theta_Y(\eta r_2) < \lambda - \eta^2 \text{ for any } Y \in Q_{2r_2}(Y^2_i) \} \\
        F^2 &= \{Y^2_i\} \setminus E^2.
    \end{split}
    \end{equation*}
    
We can iterate this process to obtain the covering
\begin{equation*}
    Q_r(X_0) \subset \cN^k \cup \bigcup_{e \in E^k} Q_{r_e}(X_e) \cup \bigcup_{Y^k_i \in F^k} Q_{r_k}(Y^k_i)
\end{equation*}
where 
\begin{equation*}
\begin{split}
    E^k &= E^{k-1} \cup \{Y^k_i: \Theta_Y(\eta r_k) < \lambda -\eta^2 \text{ for any } Y\in Q_{2r_k}(Y^k_i) \} \\
    F^k &= \{Y^k_i\} \setminus E^k
\end{split}
\end{equation*}
and $$\cN^k = Q_{2r}(X_0) \setminus \Big( \bigcup_{e\in E^k} \Bar{Q}_{r_e}(X_e) \cup \bigcup_{Y^k_i \in F^k} \Bar{Q}_{r_k}(Y^k_i) \Big)$$
is a $(\delta,r)$-cylindrical region with $\cC^k = E^k \cup F^k$. 

Now let $k \to \infty$ and thus $r_k \to 0$ and $F^k \to \cC_0$ in the Hausdorff metric. Moreover, for each $Y \in \cC_0$, we have $\lim_{s\to 0} \Theta_Y(s) \ge \lambda- \eta^2$. By Proposition \ref{p:propagation-cylinder}, $Y$ is $(\eta',s)$-cylindrical with respect to the fixed $\cL_Y = (y + \cL_Y, t_Y)$ for any $s\in(0, r]$. In particular, $Y$ is a singular point.  Hence we obtain the covering
\begin{equation*}
    Q_r(X_0) \subset \cN^k \cup \cC_0 \cup \bigcup_{e\in E}Q_{r_e}(X_e),
\end{equation*}
where $E = \cup_k E^k$ and $\cN = Q_{2r}(X_0) \setminus \big( \bigcup_{e\in E}Q_{r_e}(X_e) \cup \cC_0)$. We check that $\cN$ is a $(\delta,r)$-cylindrical region with $\cC = E \cup \cC_0$. Again, the first two properties are satisfied by construction. We check the third one. If $Y\in E^k$, then the third property is satisfied using the fact that $\cN^k$ is a $(\delta,r)$-cylindrical region and Proposition \ref{p:almost cylindrical rigidity}. We check the point $Y \in \cC_0$. Take any $r_k>r_Y=0$. Then there exists some $Y^k_i \in \cC^k$ such that $d(Y^k_i,Y) \le r_k$. Since $\cN^k$ is a $(\delta,r)$-cylindrical region satisfying \eqref{e:bettercovering5.8} for $\cC^k$, we have $\cL_{Y^k_i}\cap Q_{2r_k}(Y^k_i) \subset Q_{\tau r_k/5}(\cC^k)$ and $\cC^k \cap Q_{2r_k}(Y^k_i) \subset Q_{\tau r_k/5}(\cL_{Y^k_i})$. By proposition \ref{p:almost cylindrical rigidity}, we have $\cL_{Y^k_i}$ is $\eta'r_k$-close to $\cL_Y$ if $\eta\le\eta(\eta')$. This proves the third property for $\cL_{Y}$ at the scale $r_k$.

By  Lemma \ref{l:cylindrical region structure} and Corollary \ref{c:measure upper bound}, we obtain the Lipschitz structure of $\cC_0$ and the $1$-content estimate. To complete the proof, we still need to recover each ball $Q_{2r_e}(X_e)$ using $\{Q_{\eta^2 r_e}(X_e^j)\}_{1 \le j \le C_0}$ with $Q_{\eta^2 r_e/10}(X_e^{j_1}) \cap Q_{\eta^2 r_e /10}(X_e^{j_2}) = \emptyset$ provided $j_1 \neq j_2$. We denote the set $\tilde{E}$ to be the set of all those center $\{X_e^j\}_{e\in E; 1\le j \le C_0}$. And we define the new radius $\tilde{r}_e^j = \eta^2 r_e$. This implies that $\Theta_Y(\eta^{-1} \tilde{r}_e) < \lambda - \eta^2$ for any $Y \in Q_{2 \tilde{r}^j_e}(X^j_e)$. And moreover we have the estimate
\begin{equation*}
    \sum_{e\in E} \sum_{j=1}^{C_0}  r_e^j \le C(\tau,\eta)r.
\end{equation*}
Hence we complete the whole proof.
\end{proof}

Next we prove the covering lemma for the balls that admits a $\lambda$-pinched point with $\lambda$ away from $1$ and $\lambda_1$.

\begin{lemma}[Covering of non-cylindrical balls]\label{l:d cover}
 For any $\delta\le \delta_0$, $r \le r_0(\cM,\delta)$ and $\eta \le \eta(\cM,\delta)$, the following holds. Let $X_0$ with $t_{X_0}\ge 10\eta^2 r^{-2}$. Assume that $\sup_{Y\in Q_{2r}(X_0)}\Theta_Y(\eta^{-1}r) \le \lambda$ and that there exists some $X \in \cV_{\eta,r}(X_0;\lambda)$ for some $\lambda \notin  [1,1+ \frac{1}{2}\lambda_{gap}]\cup[\lambda_1 -\frac{1}{2}\lambda_{gap}, \lambda_1 + \frac{1}{2}\lambda_{gap}]$. Then we have the following decomposition
    \begin{equation*}
        Q_r(X_0) \subset \{*\}  \cup \bigcup_{e \in E} Q_{r_e}(X_e),
    \end{equation*}
    where $\{*\}\subset \cS_0$ contains at most one point and that any $Y \in Q_{2r_e}(X_e)$ satisfies that $\Theta_Y(\eta^{-1}r_e) < \lambda - \eta^2$. Moreover, there exists some constant $C(\eta)$ such that
    \begin{equation*}
        \sum_{e\in E} r_e  \le C(\eta) r.
    \end{equation*}
\end{lemma}

\begin{proof}
  Let $\tau\le 10^{-1}$ be fixed later. We begin by covering $Q_r(X_0)$ with balls $\{Q_{r_1}(X_i)\}$ of radius $r_1 =  \tau r$ satisfying $Q_{r_1/10}(X_i) \cap Q_{r_1/10}(X_j) = \emptyset$ as long as $X_i \neq X_j$. We then classify each ball $Q_{r_1}(X_i)$ as one of two types to obtain the following covering
    \begin{equation*}
        Q_r(X_0) \subset \bigcup_{e\in E^1} Q_{r_e}(X_e) \cup \bigcup_{f \in F^1} Q_{r_f}(X_f),
    \end{equation*}
where $r_e = r_f= r_1$. And each $X_e$ satisfies that $\Theta_Y(\eta r_e) < \lambda - \eta^2$ for any $Y \in Q_{2r_e}(X_e)$, while each $X_f$ satisfies that there exists some $X^1_f \in Q_{2r_f}(X_f)$ such that $\Theta_{X^1_f}(\eta r_f) \ge \lambda - \eta^2$. Then by the disjoint pairwise of $Q_{\eta^2 r_e/10}(X_e)$, we have the cardinality $\#\{X_e\}\le C(\tau)$ and thus
\begin{equation*}
    \sum_{e\in E^1} r_1 \le C(\tau) r_1\le C_1(\tau) r. 
\end{equation*}
By Lemma \ref{l:pinched points near axis}, we know that $\cV_{\eta,r}(X_0) \subset Q_{\tau r}(X)$. This implies that $Q_{2r_f}(X_f) \cap Q_{\tau r}(X) \neq \emptyset$. Combining this with disjointness, we obtain the refined content estimates for those balls
\begin{equation*}
    \sum_{f  \in F^1} r_f \le C_2 \tau r,
\end{equation*}
where $C_2$ is a universal constant. 
We now iterate this covering procedure. For each $f$-ball $\{Q_{r_f}(X_f)\}$, it satisfies the same assumption as $Q_{r}(X_0)$. We apply the same argument with the balls with new radius $r_2 = \tau r_1 = \tau^2 r$. This results in a second-level decomposition
\begin{equation*}
        Q_r(X) \subset  \bigcup_{e \in E^2} Q_{r_e}(X^2_e) \cup \bigcup_{f \in F^2}Q_{r_f}(X_f^2),
\end{equation*}
where each $e \in E^2$ satisfies that $\Theta_Y(\eta r_e) < \lambda - \eta^2$ for any $Y \in Q_{2r_e}(X_e)$, while each $X_f$ satisfies that $\cV_{\eta,r_f}(X_f;\lambda) \neq \emptyset$. Moreover we have the estimates
\begin{equation*}
\begin{split}
    \sum_{e \in E^2} r_e &\le  C_1(\tau) \cdot( 1+ C_2 \tau) \cdot r ;\\ 
    \sum_{f \in F^2} r_f &\le (C_2 \tau)^2 \cdot r.
\end{split}
\end{equation*}

Also note that $r_f = r_2 = \tau^2 r$ at this stage. 

Proceeding inductively, after $K$ iterations, we obtain a covering
\begin{equation*}
        Q_r(X) \subset \bigcup_{e \in E^K} Q_{r_e}(X^K_e) \cup \bigcup_{f \in F^K}Q_{r_f}(X_f),
    \end{equation*}
where each $e\in E^K$ satisfies that $\Theta_Y(\eta r_e) < \lambda - \eta^2$ for any $Y \in Q_{2r_e}(X_e)$, while each $X_f$ satisfies that $\cV_{\eta,r_f}(X_f;\lambda) \neq \emptyset$. Note that each $r_f = r_K = \tau^K r$ here. Moreover we have the estimates
\begin{equation*}
\begin{split}
    \sum_{e \in E^K} r_e &\le  C_1(\tau) \cdot(\sum_{j=0}^K (C_2 \tau)^j) \cdot r ;\\
    \sum_{f \in F^K} r_f &\le (C_2 \tau)^K \cdot r.
\end{split}
\end{equation*}

Take $E = \bigcup_i E^i$. We can fixed $\tau = (10 C_2)^{-1}$ so that the geometric series converges $\sum_{j=0}^{\infty} (C_2 \tau)^j < 2$. This implies the uniform bound for all $e$-balls: 
\begin{equation*}
    \sum_{e\in E} r_e \le 2 C_1(\tau) r.
\end{equation*}

From our construction, for any $e\in E$ we have $\Theta_Y(\eta r_e) < \lambda -\eta^2$ for any $Y \in Q_{2r_e}(X_e)$. For each such $Q_{r_e}(X_e)$, we can cover it using $C_3(\eta)$ many balls $\{Q_{\eta^2 r_e}(X_e^j)\}_{1 \le j \le C_2}$ with $Q_{\eta^2 r_e/10}(X_e^{j_1}) \cap Q_{\eta^2 r_e /10}(X_e^{j_2}) = \emptyset$ provided $j_1 \neq j_2$. We denote the set $\tilde{E}$ to be the set of all those center $\{X_e^j\}_{e\in E; 1\le j \le C_3}$. And we define the new radius $\tilde{r}_e^j = \eta^2 r_e$. This implies that $\Theta_Y(\eta^{-1} \tilde{r}_e) < \lambda - \eta^2$ for any $Y \in Q_{2 \tilde{r}^j_e}(X^j_e)$. And moreover we have the estimate
\begin{equation*}
    \sum_{e\in E} \sum_{j=1}^{C_3}  r_e^j \le C_4(\eta)r.
\end{equation*}

Moreover, the centers of the $f$-type balls converge to a subset $S \subset Q_{2r}(X_0)$ in the Hausdorff sense. By definition of $f$-balls, we know for each point $Y \in S$, there exists a sequence of points $X^i_f \subset \cV_{\eta,r_k}(X_f;\lambda)$ converging to $Y$. This implies that $\liminf_{s\to 0}\Theta_Y(s) \ge \lambda -\eta^2.$ Hence $Y \in \cS^0$. We claim that $S$ contains at most one point. Suppose there exist $Y_1, Y_2 \in S$ with $d(Y_1,Y_2) = s$. We consider the ball $Q_{2s}(Y_1)$. Then $Y_1,Y_2 \in \cV_{\eta,s}(Y_1;\lambda)$. However, by Lemma \ref{l:pinched points near axis}, we have $\cV_{\eta,s}(Y_1;\lambda) \subset Q_{\tau s}(Y_1)$. Then contradiction arises.    
\end{proof}

Finally we note that by applying the above two covering lemmas finitely many times, we can conclude the space-time covering theorem.

\begin{proof}[Proof of Theorem \ref{t:spacetime cover}]
    By Huisken's monotonicity theorem, there exists some $\Lambda$ such that $\Theta_Y(\eta^{-1}R) \le \Lambda$ for all $Y \in Q_{2R}(X)$. Suppose $\Lambda > \lambda_1 + \lambda_{gap}$. We suppose $\cV_{\eta,R}(X;\lambda) \neq \emptyset$. Otherwise we consider the set $\cV_{\eta,R}(X;\Lambda-\eta^2)$.  

    Since $\Lambda > \lambda_1 + \lambda_{gap}$, we can apply lemma \ref{l:d cover} to conclude the covering
    \begin{equation*}
        Q_R(X) \subset \bigcup_{e\in E^1} Q_{r_e}(X_e) \cup \{*\}
    \end{equation*}
    with the estimate $\sum_{e\in E^1} r_e \le C_1(\tau,\eta)R$, where $\{*\} \subset \cS^0$ contains at most one point and that for $e\in E^1$ we have that $\Theta_Y(\eta^{-1} r_e) < \Lambda-\eta^2$ for any $Y \in Q_{2r_e}(X_e)$. Then we consider the $(\Lambda-\eta^2)$-pinched set $\cV_{\eta,r_e}(X_e;\Lambda-\eta^2)$ in each ball $Q_{r_e}(X_e)$. Suppose $\Lambda-\eta^2 > \lambda_1 + \lambda_{gap}$. We can apply lemma \ref{l:d cover} again to obtain the covering
    \begin{equation*}
        Q_R(X) \subset \bigcup_{e \in E^2} Q_{r_e}(X_e) \cup S_1
    \end{equation*}
    with the estimate $\sum_{e\in E^1} r_e \le C_1(\tau,\eta)^2 R$, where $S_1 \subset \cS^0$ is countable and that for each $e\in E^2$ we have $\Theta_Y(\eta^{-1}r_e) < \Lambda-2\eta^2$ for any $Y\in Q_{2e}(X_e)$. Then we consider the $(\Lambda-2\eta^2)$-pinched set $\cV_{\eta,r_e}(X_e;\Lambda-2\eta^2)$ for each $e\in E^2$ and iterate the argument for at most $K_1 \le [\eta^{-2}(\Lambda - \lambda_1)]$ times. Indeed, we iterate the argument until we obtain such covering
    \begin{equation*}
        Q_R(X) \subset \bigcup_{e\in E^{K_1}} Q_{r_e}(X_{r_e}) \cup S_{K_1},
    \end{equation*}
    with the estimate $\sum_{e\in E^{K_1}} r_e \le C_1(\tau,\eta)^{K_1} R$, where $S_{K_1} \subset \cS^0$ is countable and that for each $e \in E^{K_1}$ we have $\Theta_Y(\eta^{-1}r_e) < \Lambda - {K_1}\eta^{2} \approx \lambda_1 + \lambda_{gap}$ for any $Y\in Q_{2r_e}(X_e)$. Then instead of applying lemma \ref{l:d cover} we apply lemma \ref{l:c cover} to each $Q_{r_e}(X_e)$. Therefore we obtain the following covering
    \begin{equation*}
        Q_R(X) \subset \bigcup_{a\in A^1} \big( (\cC_{a,0} \cup \cN_a) \cap Q_{r_a}(X_a) \big) \cup \bigcup_{e \in E^{K_1+1}} Q_{r_e}(X_e) \cup S_{K_1}
    \end{equation*}
with the estimates $\sum_{a \in A^1} r_a +\sum_{e\in E^{K_1+1} } r_e + \cH^1_P(\cup_a \cC_{0,a}) \le C_1(\tau,\eta)^{K_1} C_2(\tau,\eta) \cdot R$, where each $\cN_a \subset Q_{2r_a}(X_a)$ is a $(\delta,r_a)$-cylindrical region and that $\Theta_Y(\eta^{-1}r_e) < \lambda_1 +\lambda_{gap}- \eta^{2} $ for any $Y \in Q_{2r_e}(X_e)$. 

Similarly, we can iterate the argument and apply lemma \ref{l:c cover} for at most $K_2 \le [2\eta^{-2} \lambda_{gap}]$ many times to obtain
\begin{equation*}
        Q_R(X) \subset \bigcup_{a\in A^{K_2}} \big( (\cC_{a,0} \cup \cN_a) \cap Q_{r_a}(X_a) \big) \cup \bigcup_{e \in E^{K_1+K_2}} Q_{r_e}(X_e) \cup S_{K_1}
    \end{equation*}
with the estimates $\sum_{a \in A^{K_2}} r_a +\sum_{e\in E^{K_1+K_2} } r_e + \cH^1_P(\cup_a \cC_{0,a}) \le C_1(\tau,\eta)^{K_1} C_2(\tau,\eta)^{K_2} \cdot R$, where each $\cN_a \subset Q_{2r_a}(X_a)$ is a $(\delta,r_a)$-cylindrical region and that $\Theta_Y(\eta^{-1}r_e) < \lambda_1 +\lambda_{gap}- K_2\eta^{2} \approx \lambda_1 - \lambda_{gap}$ for any $Y \in Q_{2r_e}(X_e)$. 

Then we can apply lemma \ref{l:d cover} again, for at most $K_3 = [\eta^{-2}(\lambda_1-1-2\lambda_{gap})]$ times, to obtain the covering
\begin{equation*}
        Q_R(X) \subset \bigcup_{a\in A^{K_2}} \big( (\cC_{a,0} \cup \cN_a) \cap Q_{r_a}(X_a) \big) \cup \bigcup_{e \in E^{K_1+K_2+K_3}} Q_{r_e}(X_e) \cup S_{K_1+K_3}
    \end{equation*}
where
\begin{enumerate}
    \item $\sum_{a \in A^{K_2}} r_a +\sum_{e\in E^{K_1+K_2+K_3} } r_e + \cH^1_P(\cup_a \cC_{0,a}) \le C_1(\tau,\eta)^{K_1+K_3} C_2(\tau,\eta)^{K_2}  \cdot R$;
    \item $S_{K_1+K_3} \subset \cS^0$ is countable;
    \item each $\cC_{a,0}$ is contained in an embedded Lipschitz submanifold of dimension 1;
    \item each $\cN_a \subset Q_{2r_a}(X_a)$ is a $(\delta,r_a)$-cylindrical region;
    \item $\Theta_Y(\eta^{-1}r_e) < 1 +\lambda_{gap}$ for any $Y \in Q_{2r_e}(X_e)$. 
\end{enumerate}

Then by Proposition \ref{p:almost regular rigidity}, we have each $X_e$ is $(\delta,r_e)$-regular if $\eta$ is small enough. Hence the proof will be finished if we can prove that $\cS \subset S_{K_1+K_3} \cup \bigcup_a \cC_{a,0}$. It suffices to prove that $\cN_a\cap \cM$ is regular. 

Pick any $Y \in \cN_a \cap \cM_{t_0}$. Let $\cC_a$ be the set of centers of the cylindrical region $\cN_a\subset Q_{2r_a}(X_a)$. Since $\cC_a$ is closed, hence by definition there exists some $X=(x,t_X) \in \cC_a$ such that $d_Y \equiv d(Y,\cC) = d(Y,X) \ge r_X$. We claim that
\begin{equation}\label{e:rapid clear}
    t_0 \le t_X - \tau^2 d_Y^2.
\end{equation}

Since $d_Y \ge r_X$, then $X$ is $(\delta, d_Y)$-cylindrical by definition. By rapid clearing out lemma in \cite{CM16} we conclude that $Q_{d_Y}(X) \cap \{Z=(z,t_Z) \in \cM: t_Z-t_X \ge \tau^2 d_Y^2\}  = \emptyset$, provided $\delta$ small enough. This implies that $t_0 \le t_X + \tau^2 d_Y^2$. Suppose $|t_0 - t_X| \le \tau^2 d_Y^2<d_Y^2$. By choosing $\delta$ and $\tau$ small enough, we have $d(Y, \cL_X) \le 10^{-10}d_Y$. However, since $\cL_X \cap Q_{d_Y}(X) \subset Q_{\tau d_Y}(\cC)$, this implies that $d(Y,\cC) \le d(Y,\cL_X) + \tau d_Y < d_Y$. This contradicts that $ d(Y,\cC) = d_Y$. Hence we have $t_0 \le t_X - \tau^2 d_Y^2$. This proves the claim. 

Combining that $X$ is $(\delta,d_Y)$-cylindrical implies that $Y$ is on the $C^{2,\alpha}$ graph over the round cylinder with graph norm smaller than $\delta'$ after suitable rescaling. This proves that $Y$ is regular. Therefore we complete the proof of Theorem \ref{t:spacetime cover}.
\end{proof}

Finally we use space-time covering theorem \ref{t:spacetime cover} to prove the time-slice covering theorem \ref{t:slice covering}.

\begin{proof}[Proof of Theorem \ref{t:slice covering}]
    Write $X=(x,t_0)$. We apply the space-time covering theorem \ref{t:spacetime cover} with $R, \tau,\eta$ and $\delta'$ to the ball $Q_R(X)$ to obtain
    \begin{equation*}
        Q_R(X_0) \cap \cM \subset \bigcup_a \big((\cN_a \cup \cC_{0,a})\cap Q_{r_a}(X_a) \big) \cup \bigcup_b Q_{r_b}(X_b) \cup S_0 
    \end{equation*}
satisfying the conditions in Theorem \ref{t:spacetime cover}. Then we define $\cC_{0,a;t_0} = \cC_{0,a} \cap  \cM_{t_0}$ and $S_{0,t_0} = S_0 \cap  \cM_{t_0}$, which satisfies the requirements. Also for each $X_b=(x_b,t_{X,b})$ with $Q_{r_b}(X_b) \cap \cM_{t_0} \neq \emptyset$, we have $(x_b,t)$ is $(\delta,r_b)$-regular since each $X_b$ is $(\delta',r_b)$-regular. Therefore the balls $B_{r_b}(x_b)$ is exactly what we want. 

It remains to cover the cylindrical region $\cN_{a} \cap Q_{r_a}(X_a)$. Fix $a$. We are done if $\cN_a \cap \cM_{t_0} = \emptyset$. Let us assume $\cN_a \cap \cM_{t_0} \neq \emptyset$ and take $Y \in \cN_a \cap \cM_{t_0}$. Let $\cC_a$ be the set of centers of the cylindrical region $\cN_a\subset Q_{2r_a}(X_a)$. Since $\cC_a$ is closed, hence by definition there exists some $X=(x,t_X) \in \cC_a$ such that $d_Y \equiv d(Y,\cC) = d(Y,X) \ge r_X$. According to \eqref{e:rapid clear} we know $t_0 \le t_X - \tau^2 d_Y^2$. Since $X$ is $(\delta',d_Y)$-cylindrical by definition, it implies that $Y$ is on the $C^{2,\alpha}$ graph over the round cylinder with graph norm smaller than $\delta'$ under the rescaling at scale $[\tau d_Y, d_Y]$.

By Vitali covering lemma, there exists a covering of $\cN_a \cap \cM_{t_0}$ by $\{B_{100d_i}(y_i)\}$ such that
\begin{enumerate}
    \item $d_i = d((y_i,t_0), \cC_a) = d((y_i,t_0),X_i)$ with $X_i = (x_i,t_{X_i}) \in \cC_a$;
    \item $\{B_{10d_i}(y_i)\}$ is pairwise disjoint.
\end{enumerate}
By definition we have $r_{X_i} \ge d_i$, hence $X_i$ is $(\delta',d_i)$-cylindrical. Since $t_0 \le t_{X_i} - \tau^2 d_i^2$, each piece $B_{100d_i}(y_i) \cap \cM_{t_0}$ is a $C^{2,\alpha}$ graph over a round cylinder $\RR \times \mathds{S}^1({\sqrt{2}})$ with graph norm smaller than $\delta'$ under the scaling at scale $|t_{X_i}-t_0|^{1/2} \in [\tau d_i, d_i]$. Hence by the standard covering on the cylinder we obtain 
\begin{equation*}
    B_{100 d_i}(y_i)  \cap \cM_{t_0} \subset \bigcup_{j} B_{r_{i;j}}(y_{i;j})
\end{equation*}
where
\begin{enumerate}
    \item each $(y_{i;j},t_0)$ is $(\delta,r_{i;j})$-regular;
    \item $\sum_{j} r_{i;j} \le C(\tau) d_i$.
\end{enumerate}

Recall that $\{Q_{10d_i}(Y_i),Y_i=(y_i,t_0)\}$ is pairwise disjoint. Since $Q_{2d_i}(X_i) \subset Q_{10d_i}(Y_i)$, hence $\{Q_{2d_i}(X_i)\}$ is also pairwise disjoint. By the structure lemma \ref{l:cylindrical region structure}, there exists some bi-Lipschitz map $\pi_X : \cC \to \cL_X$. Hence $\{\pi(Q_{d_i}(X_i))\}$ is pairwise disjoint in $\cL_X$. Therefore,
\begin{equation*}
    \sum_{i} d_i \le  C r_a.  
\end{equation*}
This implies that
\begin{equation*}
    \sum_i \sum_j r_{i;j} \le C'(\tau) r_a.
\end{equation*}
We apply this covering to each piece $\cN_a\cap \cM_{t_0}$. It remains to check the 1-content estimate. Combining the estimate $\sum_a r_a + \sum_b r_b \le C(\cM,\delta',\epsilon,\eta)R$ from Theorem \ref{t:spacetime cover} implies the $1$-content estimate \eqref{e:slice 1 content}. This completes the proof.
\end{proof}

\section{Proof of Main theorems}

In this section, we prove the main Theorem \ref{t:main almost regular}. The proof for Theorem \ref{t:main smooth} is verbatim.  

\begin{proof}[Proof of Theorem \ref{t:main almost regular}:]

Let $\cM_{t} \subset \RR^3 \times [0,T]$ be a bounded almost regular flow starting from a smooth closed surface $M_0 \subset \RR^3$. Let $T_0>0$ be the first singular time, i.e. $\cM_{t=t_0}$ is smooth for any $t_0 < T_0$. 

We can pick $\delta,\tau,r_0$ and $\eta$ small such that Theorem \ref{t:slice covering} and \ref{t:spacetime cover} hold. 

Since $\cM|_{0\le t \le T_0/2}$ is smooth, there exists some constant $C_0(\cM)$ such that the following estimates hold for any $0\le t \le T_0/2$:
\begin{equation}
    \int_{\cM_t} |A|  \le C_0 \quad \text{ and } \quad D_{int}(\cM_t) \le C_0.
\end{equation}

Then for any $T_0/2 < t \le T$, consider a Vitali covering of the slice $\cM_t$ by $B_{\eta^2 T_0/10}(x_k)$. Then the number of those balls in the covering is bounded by some constant $C_1$ depending only on the extrinsic diameter of $\cM$ and $T_0$. We can apply the time-slice covering Theorem \ref{t:slice covering} to each $B_{\eta^2T_0/10}$ and then combine them together to obtain the covering
\begin{equation*}
    \cM_t \subset \bigcup_{i=1}^{\infty} B_{r_i}(x_i) \cup \bigcup_{j=1}^{\infty} \cC_{0,j;t_0} \cup S_{0;t_0},
\end{equation*}
where each $(x_i,t)$ is $(\delta,r_i)$-regular with uniform 1-content estimate: 
\begin{equation*}
    \sum_i r_i \le C(\cM).
\end{equation*}

The covering theorem implies the structure of the singular set $\cS_t$ directly. Next we prove the diameter and curvature bound. Since $\sup_{B_{\delta^{-1}r_i}} r_i|A| \le \delta^2$. Therefore, by Gauss equation we have $D_{int}(B_{r_i}(x_i) \cap \cM_t \setminus\cS_t) \le 2 r_i$ and $Area(B_{r_i}(x_i) \cap \cM_t \setminus \cS_t) \le C_0 r_i^2$. Therefore we have
\begin{equation*}
    D_{int}(\cM_t \setminus\cS_t) \le \sum_i D_{int}(B_{r_i}(x_i) \cap \cM_t \setminus\cS_t) \le 2\sum_i r_i \le 2 C(\cM),
\end{equation*}
and
\begin{equation*}
    \int_{\cM_t\setminus\cS_t} |A| \le \sum_i \int_{B_{r_i}(x_i) \cap \cM_t \setminus \cS_t} |A| \le \delta^2 \sum_i r_i^{-1} \cdot C_0r_i^2 \le C_0 \cdot \delta^2 \cdot C(\cM).
\end{equation*}

This completes the proof of estimates at a single time-slice. 

Next we prove the estimates on the space-time region. Consider a Vitali covering of $\cM_{T_0/2 \le t \le T}$ by $Q_{\eta^2T_0/10}(X_i)$. Then the number of those balls in the covering is bounded by some constant $C_2(\cM)$ depending only on the extrinsic diameter of $\cM$, $T_0$ and $T$. Now we can apply the space-time covering Theorem \ref{t:spacetime cover} to each individual ball $Q_{\eta^2T_0/10}(X_i)$ and then combine them together to obtain the covering
\begin{equation*}
    \cM_{T_0/2 \le t \le T} \subset \bigcup_a \big((\cN_a \cup \cC_{0,a})\cap Q_{r_a}(X_a) \big) \cup \bigcup_b Q_{r_b}(X_b) \cup S_0.
\end{equation*}
where each $\cN_a \subset Q_{2r_a}(X_a)$ is a $(\delta,r_a)$-cylindrical region, each $X_b$ is $(\delta,r_b)$-regular. Moreover we have the uniform 1-content estimate
\begin{equation}
    \sum_a r_a + \sum_b r_b + \cH_P^1(S_0 \cup \bigcup_a\cC_{0,a}) \le C_1.
\end{equation}

The space-time covering theorem implies the Lipschitz structure of the singular set $\cS$ and the finiteness of the measure $\cS$.

It remains to prove the weak $L^3$ bound for $A$. Since $\cM_{0\le t\le T_0/2}$ is smooth, we have $|A|(X) \le C_0(\cM)$ for any $X\in \cM_{0\le t\le T_0/2}$. Define $\Omega_s \equiv \{X \in \cM: |A|(X) > s^{-1}\} \cap \cM_{T_0/2\le t  \le T}$. 

Let $Y \in \cN_a\cap Q_{r_a}(X_a)$. Let $\cC_a$ be the set of centers of the cylindrical region $\cN_a\subset Q_{2r_a}(X_a)$. Since $\cC_a$ is closed, hence by definition there exists some $X=(x,t_X) \in \cC_a$ such that $d_Y \equiv d(Y,\cC) = d(Y,X) \ge r_X$. According to \eqref{e:rapid clear} we know $t_0 \le t_X - \tau^2 d_Y^2$. Since $X$ is $(\delta,d_Y)$-cylindrical by definition, it implies that $Y$ is on the $C^{2,\alpha}$ graph over the round cylinder with graph norm smaller than $\delta$ under the rescaling at scale $[\tau d_Y, d_Y]$. This implies that $|A|(Y) \le C(\tau)d_Y^{-1}$. Hence we have $\Omega_s \cap \cN_a \subset Q_{C(\tau)s}(\cC_a)$. By Lemma \ref{l:volume of center}, we have 
\begin{equation*}
    \mu_{\cM}(\Omega_s \cap \cN_a) \le \mu_{\cM}(Q_{C(\tau) s}(\cC_a)) \le C(\cM,\tau) s^3 r_a. 
\end{equation*}
Actually, if $ r_a<s$, we can simply choose $\Omega_s\cap \cN_a\subset Q_{r_a}(X_a)$ and one can get from Lemma \ref{l:volume of center} that $\mu_{\cM}(\cN_a)\le Cr_a^4\le Cs^3r_a$.  Combining $\sum_a r_a \le C_1$ implies that 
\begin{equation*}
    \mu_{\cM}(\Omega_s \cap (\bigcup_a \cN_a)) \le C_1 C(\cM,\tau) s^3.
\end{equation*}

Moreover, since in each ball $Q_{r_b}(X_b)$, we have $\sup_{Q_{\delta^{-1} r_b}} |A| \le \delta^2 r_b^{-1}$, hence
\begin{equation*}
    \Omega_s \setminus \bigcup_a \cN_a \subset \bigcup_{r_b \le \delta^2 s} Q_{r_b}(X_b) \cap \cM.
\end{equation*}

Combining the local area estimate for $\cM$ and the 1-content estimate for $r_i$ implies that
\begin{equation*}
    \mu_{\cM }(\Omega_s \setminus \bigcup_a \cN_a) \le \sum_{r_b \le \delta^2 s} \mu_{\cM}(Q_{r_b}(X_b)) \le C(\cM)  \sum_{r_b \le \delta^2 s}  r_b^4 \le C(\cM,\delta) s^3.
\end{equation*}

This completes the proof.    
\end{proof}

\end{document}